\documentclass[oneside,english]{amsart}
\usepackage[T1]{fontenc}
\usepackage[latin9]{inputenc}
\usepackage{float}
\usepackage{amsthm}
\usepackage{amstext}
\usepackage{amssymb}
\usepackage{graphicx}
\usepackage{esint}
\usepackage[shortlabels]{enumitem}

\makeatletter
\numberwithin{equation}{section}
\numberwithin{figure}{section}
\theoremstyle{plain}
\newtheorem{thm}{\protect\theoremname}[section]
  \theoremstyle{definition}
  \newtheorem{rem}[thm]{\protect\remarkname}
  \theoremstyle{definition}
  \newtheorem{defn}[thm]{\protect\definitionname}
  \theoremstyle{plain}
  \newtheorem{prop}[thm]{\protect\propositionname}
  \theoremstyle{plain}
  \newtheorem{example}[thm]{\protect\examplename}
  \theoremstyle{plain}
  \newtheorem{lem}[thm]{\protect\lemmaname}
 \newlist{casenv}{enumerate}{4}
 \setlist[casenv]{leftmargin=*,align=left,widest={iiii}}
 \setlist[casenv,1]{label={{\itshape\ \casename} \arabic*.},ref=\arabic*}
 \setlist[casenv,2]{label={{\itshape\ \casename} \roman*.},ref=\roman*}
 \setlist[casenv,3]{label={{\itshape\ \casename\ \alph*.}},ref=\alph*}
 \setlist[casenv,4]{label={{\itshape\ \casename} \arabic*.},ref=\arabic*}
  \theoremstyle{plain}
  \newtheorem{cor}[thm]{\protect\corollaryname}

 \newtheorem{maintheorem}{Theorem}

\usepackage{color,graphics}

\AtBeginDocument{
  
}

\makeatother

\usepackage{babel}
  \providecommand{\corollaryname}{Corollary}
  \providecommand{\definitionname}{Definition}
  \providecommand{\lemmaname}{Lemma}
  \providecommand{\propositionname}{Proposition}
  \providecommand{\remarkname}{Remark}
 \providecommand{\casename}{Case}
\providecommand{\theoremname}{Theorem}
\providecommand{\examplename}{Example}

\global\long\def\d{\textup{d}}
\global\long\def\FF{\mathcal{F}}
\global\long\def\DD{\mathcal{D}}
\global\long\def\MM{\mathcal{M}}
\global\long\def\ord{\textup{ord}}
\global\long\def\Fh{\widehat{\mathcal{F}}}
\global\long\def\Gh{\widehat{\mathcal{G}}}

\global\long\def\var{\Delta}

\newcommand{\cl}[1]{\mathcal{#1}}

\newcommand{\sing}{\rm Sing}

\newcommand{\val}{{\rm   v}}

\global\long\def\dd{\textup{d}}

\begin{document}

\title{The Poincaré problem  in the dicritical case}

\author{Yohann Genzmer \& Rogério Mol}

\maketitle

\begin{abstract} We develop  a study on local polar invariants of planar complex analytic foliations at
$(\mathbb{C}^{2},0)$, which leads to the characterization of second type foliations and of generalized curve foliations, as well as a description of the $GSV$-index.  We apply it to the Poincaré problem for foliations on the complex projective plane  $\mathbb{P}^{2}_{\mathbb{C}}$, establishing, in the dicritical case, conditions for the
existence of a  bound  for the degree of an invariant algebraic curve $S$ in terms of the degree of the foliation $\mathcal{F}$. We characterize the existence of a solution for the Poincaré problem in terms of the structure of the set of local separatrices of $\mathcal{F}$ over the curve  $S$. Our method, in particular, recovers the known solution for the non-dicritical case, $\deg(S) \leq \deg (\mathcal{F}) + 2$.

\addtocontents{toc}{\protect\setcounter{tocdepth}{1}}
\tableofcontents{}
\end{abstract}
\footnotetext[3]{Work supported by MATH-AmSud Project CNRS/CAPES/Concytec. First author supported by a grant ANR-13-JS01-0002-0. Second author supported by  Pronex/FAPERJ and Universal/CNPq.}

\section{Introduction}

Let $\mathcal{F}$ be a singular holomorphic foliation on $\mathbb{P}^{2}_{\mathbb{C}}$.   The  number  of points of
tangency, with multiplicities counted, between $\mathcal{F}$ and a  non-invariant line
$L \subset  \mathbb{P}^{2}_{\mathbb{C}}$  is   the  \emph{degree} of the foliation  and is
denoted by $\deg(\mathcal{F})$. In  \cite{poincare1891}, H. Poincaré proposed the problem of bounding  the
  degree of
an  algebraic curve $S$ invariant by  $\mathcal{F}$ in  terms of
$\deg(\mathcal{F})$    as a step in
finding a rational first integral for a polynomial differential equation in two complex variables. Known in
Foliation Theory as the \emph{Poincaré problem}, along the past few decades this problem has  gained
some partial answers. In 1991,
D. Cerveau and A. Lins Neto proved in \cite{cerveau1991} that if $S$
has at most nodal singularities then  $\deg(S) \leq \deg(\mathcal{F}) +
2$,   this bound being reached if and only if
$\mathcal{F}$ is a logarithmic foliation ---  one induced by a closed meromorphic 1-form
with simple poles. Later, in 1994, M. Carnicer obtained in   \cite{carnicer1994}
 the same inequality when  the singularities of
$\mathcal{F}$   over $S$  are all \emph{non-dicritical}, meaning that the number of local \emph{separatrices}
--- local irreducible invariant
 curves ---
is finite.  In 1997, in the works \cite{brunella1997I} and
 \cite{brunella1997II}, M. Brunella formulated the Poincaré problem in terms of $GSV$-indices,
defined by X. Gómez-Mont J. Seade and A. Verjovsky in \cite{gomezmont1991} as a kind of Poincaré-Hopf index of the
  restriction to an invariant curve of  a vector field tangent to $\mathcal{F}$.
To wit, Brunella shows that the bound  $\deg(S) \leq
\deg(\mathcal{F}) + 2$ occurs whenever the sum over $S$ of the $GSV$-indices of $\mathcal{F}$ with respect to the local branches of $S$ is non-negative.

Evidently, the study of global invariant curves leads us to the universe of local foliations on $(\mathbb{C}^{2},0)$, in which
we distinguish two families    with relevant properties.
 First,  \emph{generalized curve foliations}, defined in \cite{camacho1984} by C. Camacho, A. Lins Neto and P. Sad, which are foliations without saddle-nodes in their
 desingularization. They  have a property of minimization of Milnor numbers and are characterized, in the non-dicritical case, by the vanishing of the  $GSV$-index  \cite{brunella1997II,lehmann2001}.
The second family, which contains the first one, is formed by
 {\em  second type foliations}, introduced by J.-F. Mattei and E. Salem in
\cite{mattei2004}, which
may admit saddles-nodes when desingularized provided that they are
non-tangent saddles-nodes, meaning that no weak separatrix is  contained in the
desingularization divisor. They are characterized by the fact that their
desingularization coincide with the reduction of the set of formal separatrices.
These foliations satisfy a property of minimization of
the algebraic multiplicity    \cite{mattei2004, genzmer2007}.

In a recent work \cite{cano2015}, F. Cano, N. Corral and the second author developed a study of local polar
invariants obtaining,  in the non-dicritical case, a characterization of generalized curves and
of foliations of second type as well as an expression of the $GSV$-index in terms of these invariants. Essentially, the technique therein consists on calculating the intersection number between a generic polar curve of a local foliation $\mathcal{F}$
and a curve of formal separatrices $C$. The same  number is produced for the formal ``reference foliation'' having the
local equation of $C$ as a first integral. The difference of these two numbers is the $GSV$-index of
$\mathcal{F}$ with respect to $C$. In this way,   the known answers to the Poincaré problem just mentioned are obtained.

In this paper we  extend this approach    to   \emph{dicritical} foliations --- those with infinitely many
separatrices. The difficulty now lies
in choosing a finite set of separatrices in order to produce such a ``reference foliation''. The solution is to use a \emph{balanced equation of separatrices},
a concept introduced by the first author in  \cite{genzmer2007} for the study of the ``realization problem'' ---
 the existence of   foliations with prescribed   reduction of
singularities and   projective holonomy representations.
Given a local foliation $\Fh$ at $(\mathbb{C}^{2},0)$ with minimal reduction of singularities $E:(\MM,\DD)\rightarrow(\mathbb{C}^{2},0)$, an irreducible component $D \subset \DD$ is said to be
\emph{non-dicritical} --- respectively \emph{ dicritical} --- if it is invariant --- respectively non-invariant ---
by the strict transform foliation $E^{*} \Fh$.  The \emph{valence} of $D \subset \DD$ is the number $\val(D)$  of other components of $\DD$ intersecting $D$.
A balanced equation of separatrices turns out to be
a formal meromorphic function $\hat{F}$ that  encompasses the equations of  all
isolated separatrices --- the ones  crossing non-dicritical components of  $\DD$ ---
along with the equations of $2 - \val(D)$  separatrices associated to each dicritical component $D \subset \DD$.
This can be a negative number, so dicritical separatrices may engender poles in the balanced equation.
We can additionally  adjust this definition when  a  local set of separatrices $C$ for $\Fh$ is fixed in order to get a  balanced equation \emph{adapted}
to   $C$. This is achieved by rebalancing the number of dicritical separatrices of $\hat{F}$ in such a way that
 $C \subset (\hat{F})_{0}$. We develop these concepts in section~\ref{section-balanced-equation}.

In section~\ref{section-polar-excess}, our starting point is
the  extension of the definition
 of the \emph{polar intersection number}
  introduced in \cite{cano2015}  to  a formal meromorphic $1-$form, which will normally be $\d \hat{F}$, where $\hat{F}$ is a balanced equation of separatrices of the foliation $\Fh$.
For a fixed set of separatrices $C$ of   $\Fh$, the comparison of the polar intersection numbers of  $\Fh$ and
$\d \hat{F}$, where
 $\hat{F}$ is a balanced equation of separatrices  such that $C \subset (\hat{F})_{0}$,
gives rise to the \emph{polar excess  index}, denoted by
$\var_{p} (\Fh,C)$. This non-negative invariant   %(Proposition~\ref{prop:VarPos})
works as a measure of the existence of saddles-nodes in the desingularization of the foliation: $\var_{p} (\Fh,(F)_{0}) = 0$ if and only if
$\Fh$ is a generalized curve. This is the content of   Theorem~\ref{thm-generalized-curve}, which extends to the dicritical case the characterization of generalized curve foliation provided, in the non-dicritical case, by the vanishing of the $GSV$-index
with respect to the complete set of separatrices.
Actually, in section \ref{section-GSVindex},
Theorem~\ref{thm-gsv-polar} establishes a link between the polar excess and
the $GSV$-index  for a convergent  set of separatrices:
\[GSV_{p}(\Fh,C)=\var_{p}(\Fh,C)+(C,(\hat{F})_{0}\backslash C)-(C,(\hat{F})_{\infty}),\]
where $\hat{F}$ is a balanced equation of separatrices such that $C \subset (\hat{F})_{0}$. Notice that, when $\Fh$ is non-dicritical and   $C$ is the complete set of separatrices, this gives
    $GSV_{p}(\Fh,C)=\var_{p}(\Fh,C)$.

This formulation of the $GSV$-index enables us in Theorem~\ref{thm-poincare-dic} in section
 \ref{section-poincaré-problem} to propose a bound to the Poincaré problem in terms of local balanced sets of separatrices.  For  a foliation $\mathcal{F}$
 on $\mathbb{P}^{2}_{\mathbb{C}}$    having an  invariant algebraic curve $S$, if $d = \deg(\mathcal{F})$ and
 $d_{0} = \deg(S)$, it holds
\[  d_{0} \leq
d+2+\frac{1}{d_{0}}\sum_{p\in \sing(\mathcal{F})\cap S} \left[
(S, (\hat{F}_{p})_{\infty})_{p}-(S,(\hat{F}_{p})_{0}\backslash S)_{p} \right], \]
where $\hat{F}_{p}$ is a balanced equation   adapted
to the local branches of $S$.
Besides,
 equality holds   if all singularities
of $\mathcal{F}$ over $S$  are generalized curves.  In  particular, this inequality recovers the bound $d_{0} \leq d +2$ for the cases treated in \cite{cerveau1991} and \cite{carnicer1994}.

In the right side of the previous formula, the terms $(S, (\hat{F}_{p})_{\infty})_{p}$ %, for $p \in \sing(\mathcal{F})\cap S$
are obstructions to
the existence of a ``universal bound'' for the Poincaré problem. This is precisely what happens in two classical counterexamples to be discussed
in section  \ref{section-poincaré-problem}:
the foliations of degree 1 on $\mathbb{P}^{2}_{\mathbb{C}}$ given by $\omega= \d ( x^{p}z^{q-p}/y^{q})$, with $p < q$,
and the pencil of Lins Neto \cite{linsneto2002}, a family of foliations of degree 4 admitting rational first integrals
with unbounded degrees. In the first case, the typical fiber has a local branch at a dicritical  singularity crossing a  dicritical component of valence two.
In the second family,   the generic fiber repeatedly crosses  radial singularities at a   number of times  which is unbounded within the family.

In section~\ref{section-topologically-bounded} we study topologically bounded invariants of local foliations ---
those  bounded by a function of the Milnor number.
The central result of this section ---
Theorem~\ref{thm:Let--be-1} --- states that a    local curve of separatrices   that
contains, besides the isolated separatrices, one separatrix attached to each dicritical component of valence one and
a fixed number of separatrices attached to dicritical components of valence three or higher --- separatrices
crossing components of valence two are forbidden --- has a reduction process
whose length is topologically bounded, the same being true for
the algebraic multiplicity. This result is sharp, as shown by the example
$px\dd y-qy\dd x = 0$ with $p,q \in \mathbb{Z}^{+}$  co-prime. Here, the Milnor number is one and
curves of separatrices --- with a single branch passing by a component of valence two, when $q > p >1$, or with two branches passing by a dicritical component of valence one, when $q > p =1$ --- may be obtained
with   reduction trees of arbitrarily large length.
Returning to the Poincaré problem, in Theorem~\ref{thm-poincare-dicI} we use the inequality of Theorem~\ref{thm-poincare-dic}
in order to prove the existence of a bound for the Poincaré problem whenever
the local branches of the algebraic curve $S$ are subject to the conditions of topological
boundedness of Theorem \ref{thm-poincare-dic}.
% This generalizes the result of N. Corral and P. Fernández in \cite{corral2006}, where only isolated separatrices where admitted as local branches of $S$.
%We emphasize  that Theorem~\ref{thm-poincare-dicI}
 This result especially indicates that the two classical counterexamples for the Poincaré problem
just mentioned  offer essentially the two ways to violate the existence of a bound: either by means of  highly degenerated   separatrices crossing
 dicritical components of valences one or  two, or through a multiple branched curve of separatrices
attached to  dicritical components of other valences.

We close this article with section \ref{section-topological-invariance}, where
we  apply   local polar invariants on a result on the
topological invariance of the algebraic multiplicity of a foliation. In Theorem~\ref{thm-topological-invariance}
we prove that, for local foliations at  $(\mathbb{C}^{2},0)$ having only convergent separatrices, the
property of being second class and the algebraic multiplicity are   topological invariants.
This extends similar results in \cite{camacho1984}, for generalized curve foliations, and in
\cite{mattei2004}, for non-dicritical second class foliations.

\section{Basic notions and notations}

 A \emph{ germ of formal foliation $\Fh$} in $\mathbb{C}^{2}$
is the object defined by a germ of formal $1-$form at $0\in\mathbb{C}^{2}$
\[
\hat{\omega}=\hat{a}\left(x,y\right)\d x+\hat{b}\left(x,y\right)\d y,
\]
where $\hat{a},\hat{b}\in\mathbb{C}\left[\left[x,y\right]\right]$.
A \emph{separatrix} for $\Fh$ is a germ of formal irreducible invariant curve.  If $S$ is  defined
by a formal equation $\hat{f}$,  then the invariance condition is expressed algebraically as
\[
\hat{f}\textup{\ divides\ }\hat{\omega}\wedge\d\hat{f}\textup{\ in\ }\mathbb{C}\left[\left[x,y\right]\right].
\]
A formal foliation is said to be \emph{non-dicritical} when it has
finitely many separatrices. From now on, $E:(\MM,\DD)\rightarrow(\mathbb{C}^{2},0)$
stands for the process of \emph{reduction of singularities} or  \emph{desingularization}  of $\Fh$
(see
\cite{seidenbeg1968} and also \cite{camacho1984}), obtained by a finite sequence of punctual blow-ups having $\DD = E^{-1}(0)$ as  the exceptional divisor, formed by a finite union of projective lines with normal crossings.
In this process, all separatrices of $\Fh$ become smooth, disjoint and transverse to $D$, none of them
passing though a corner.
Besides, the
singularities along $\DD$ of the pull-back foliation $E^{*}\Fh$ become \emph{reduced} or \emph{simple}, which means
that, under some local formal change of coordinates, their linear
parts belong to the following list:
\begin{enumerate}[(i)]
\item $y\d x - \lambda x\d y$ ,\quad{}$\lambda\not\in\mathbb{Q}^{+}$; \smallskip
\item $x\d y$.
\end{enumerate}

In the first case --- which is referred to as \emph{non-degenerate}---, there are two formal separatrices tangent to the axes
$\left\{ x=0\right\} $ and $\left\{ y=0\right\}$. In the convergent
category, both separatrices converge. In the second case,
the singularity is said to be a \emph{saddle-node}. The formal normal
form for such a singularity is given in \cite{mattei2004}: up to a formal
change of coordinates, the singularity is given by a $1$-form of
the  type
\begin{equation}
\label{eq-saddle-node-form}
\left(\zeta x^{k}-k\right)y\d x+x^{k+1}\d y,\quad k\in\mathbb{N}^{*},\zeta\in\mathbb{C}.
\end{equation}
The invariant curve $\{x=0\}$ is called   \emph{strong invariant curve}
while $\{y=0\}$ is the \emph{weak} one. In the convergent category, the
strong separatrix  always converges whereas the weak separatrix
may be purely formal,  as in the famous example of Euler
\[
\left(x-y\right)\d x+x^{2}\d y,
\]
where $y\left(x\right)=\sum_{n\geq1}\left(n-1\right)!x^{n}$ is the
Taylor expansion of the weak invariant curve. The integer $k+1 >1$ will be called
\emph{weak index} of the saddle-node.

 If, in the reduction process, the germ of the divisor $\DD$
happens to contain the weak invariant curve of some saddle-node, then the
singularity is said to be a \emph{tangent saddle-node}. A non-tangent saddle-node
is also   said to be \emph{well-oriented} with respect to $\DD$. We will denote
by ${\rm T}(\Fh)$ the set of all   tangent saddle-nodes
of $\Fh$. Following \cite{mattei2004}, we say that the foliation $\Fh$ is \emph{in
the second class} or is \emph{of second type} when none of the singularities
of $E^{*}\Fh$  over $\DD$ are tangent saddle-nodes.

Let $\Fh$ be  a germ of foliation  having $(S,p)$ as  a germ of formal smooth invariant curve.
Take local coordinates $(x,y)$ in which $S$ is the curve $\{y=0\}$ and $p$ the point $(0,0)$. Let
$\omega=\hat{a}(x,y)\d x+\hat{b}(x,y)\d y$ be a defining $1$-form for $\Fh$.
The integer $\textup{ord}_{0}\hat{b}(x,0)$ is called the \emph{tangency index}
of $\Fh$ at $p$ with respect to $S$ and is denoted by $\textup{Ind}(\Fh,S,p)$.
This is an invariant associated to $\Fh$ and $S$, independent
 of the choices  made. If $S$ is weak separatrix of a saddle-node, then $\textup{Ind}(\Fh,S,p) > 1$
 is precisely the weak index. On the other hand, if $S$ is either the strong separatrix of a saddle-node or a separatrix of
 a non-degenerate reduced singularity, then
 $\textup{Ind}(\Fh,S,p) = 1$.

In the exceptional divisor $\mathcal{D}$, we denote by $\textup{Dic}(\Fh)$
the set of \emph{dicritical components}, comprising all projective lines generically
transverse to $E^{*}\Fh$.
A separatrix
$S$  of $\Fh$ is said to be \emph{isolated} if its strict transform $E^{*}S$ is not attached to a dicritical
component. This concept is well defined as long as we fix a minimal  reduction of singularities for $\Fh$ --- see that,
in the definition of  reduction of singularities, there are also conditions on the desingularization of the separatrices.
We  denote by $\mathcal{I}(\Fh)$ the set of isolated separatrices.
On the other hand, a separatrix whose strict transform
crosses a dicritical component is called a \emph{curvet}.
%, following an expression introduced in \textbackslash{}ref\{\}\emph{. }
The set of all curvets associated
to  $D\in\textup{Dic}(\Fh)$ is
 denoted by $\textup{Curv}(D)$. Finally, $\val(D)$ stands for the \emph{valence} of $D\in\textup{Dic}(\Fh)$,
 defined as the number of components of $\mathcal{D}$ intersecting  $D$.

\begin{rem}
The above definitions  can be formulated in the convergent category
and it may seem somewhat strange to introduce them in the formal category. Indeed, for convergent foliations,
the natural and geometric notion of \emph{leaves}
does exist whereas only the notion of separatrix makes sense in the
formal setting. However, the interest and the need to work with
formal objects will be evident as soon as the concept of \emph{balanced
equation} is defined.
\end{rem}

\section{Balanced equation of separatrices}

\label{section-balanced-equation}

With a slight change in the definition, we follow \cite{genzmer2007}
in the  concept below.
\begin{defn}
\label{def-balanced-set}
A \emph{balanced equation of separatrices} for a germ of formal singular
foliation in $(\mathbb{C}^{2},0)$ is a formal meromorphic
function $\hat{F}$ whose divisor
has the   form
\[ (\hat{F})_{0}-(\hat{F})_{\infty} \ = \
\sum_{C\in\mathcal{I} (\Fh)} (C)+\sum_{D\in\textup{Dic} (\Fh)}\ \sum_{C\in\textup{Curv\ensuremath{ (D)}}}a_{D,C} (C) ,
\]
where, for every  dicritical  component $D \subset \DD$,
  the coefficients $a_{D,C}\in\mathbb{Z}$  are  non-zero except for finitely many $C\in\textup{Curv\ensuremath{ (D)}}$
  and satisfy the following equality:
\begin{equation}
\sum_{C\in\textup{Curv\ensuremath{ (D)}}}a_{D,C}=2- \val(D).\label{eq:1-balance}
\end{equation}
The balanced equation of separatrices $\hat{F}$ is said to be \emph{adapted}
to a curve of separatrices $C$ if $C \subset (\hat{F})_{0}$.
\end{defn}
For instance, the radial foliation given by $x\d y-y\d x$ has
only one dicritical component whose valence is $0.$ Thus $F=xy$ is
a balanced equation of separatrices. However, $F= xy (x-y)/(x+y)$
is also  a balanced equation, adapted to the curve $C = \{x-y=0\}$. If $\Fh$ is non-dicritical, then
a balanced equation  is nothing but any equation
of the finite set of separatrices.

We  recall some basic facts about balanced equations of separatrices
established in \cite{genzmer2007}. First a definition:
\begin{defn}
Let $\Fh$ be a formal foliation at $(\mathbb{C}^{2},0)$ and
$E:(\MM,\DD)\rightarrow(\mathbb{C}^{2},0)$ be a minimal process of reduction of singularities.
The \emph{tangency excess} of $\Fh$ along $\DD$ is the number
\begin{equation}
\label{eq-tau}
\tau(\Fh)=\sum_{q \in {\rm T}(E^{*}\Fh)}\sum_{D\in V(q)}\rho(D)(\textup{Ind}(E^{*}\Fh,D,q)-1),
\end{equation}
where $q$ runs over all the tangent saddle-nodes of $E^{*}\Fh$ and  $V(q)$ stands for the set of irreducible components of
$\DD$  containing
the point $q$. The number $\rho(D)$ stands for  the \emph{multiplicity} of $D$, which coincides with
the algebraic multiplicity of a curve $\gamma$ at $(\mathbb{C}^{2},0)$ such that $E^{*} \gamma$ is transversal
to $D$  outside a corner of $\DD$.
\end{defn}
It is clear that $\tau(\Fh) \geq 0$. Besides, $\tau(\Fh) = 0$ if and only if
there are no tangent saddle-nodes in the reduction of singularities of $\Fh$, that is, if and only if $\Fh$ is a foliation of second type. Having the definition in mind, the following fact is proved in \cite{genzmer2007}:

\begin{prop}
\label{prop:Equa-Ba} Let $\hat{F}$ be a balanced equation of
separatrices for the formal foliation $\Fh$. Denote by $\nu_{0}(\hat{F})$ and $\nu_{0}(\Fh)$ their
algebraic multiplicities. Then
\[\nu_{0}(\Fh)=\nu_{0}(\hat{F})-1+\tau(\Fh).\]
%\[\nu_{0}(\hat{F})=\nu_{0}(\Fh)+1-\tau(\Fh).\]
\end{prop}

Inspired in this result, if
$\hat{F}$ is a balanced  equation of separatrices for the foliation $\Fh$,
we define the number
\[ \nu^{*}_{0}(\Fh)=\nu_{0}(\hat{F})- 1\]
and call it the \emph{pure algebraic multiplicity} of $\Fh$ at $0 \in \mathbb{C}^{2}$.
This invariant assembles the contribution to the algebraic multiplicity given by the separatrices
and the \emph{reduction structure} --- by which we mean the combinatory of the
desingularization tree along with the position of the dicritical components ---, but discarding the contribution given by the tangent saddles-nodes.
Evidently, $\nu^{*}_{0}(\Fh) \leq  \nu_{0}(\Fh)$,   equality holding if
and only if $\Fh$ is of second type.
The number $\nu^{*}_{0}(\Fh)$
would turn out to be the the
 algebraic multiplicity   of some  second type foliation sharing with
 $\Fh$ the same reduction structure and the same balanced equation. However,    we cannot assure the
 existence of  such a foliation and we do not know if $\nu^{*}_{0}(\Fh)$ is realizable as an algebraic multiplicity
 in this way.
 Nevertheless, we can prove the following:

\begin{lem}
\label{lemma_pure_multiplicity}
We have $\nu^{*}_{0}(\Fh) \geq 0$. The inequality is strict when $\Fh$ is dicritical.
\end{lem}
\begin{proof}
We have to proof that,
if $\hat{F}$ is a balanced equation of separatrices for $\Fh$, then
$\nu_{0}(\hat{F}) \geq 1$, the equality being strict in the   dicritical case.
The result is obvious if $\Fh$ is non-dicritical, since the balanced equation has no poles
and a separatrix always exits by the Separatrix Theorem \cite{camacho1982}. For the dicritical case,
let $E:(\MM,\DD)\rightarrow(\mathbb{C}^{2},0)$ be a minimal reduction process for $\Fh$.
The proof relies on the fact that through
each connected component of the invariant part of the divisor $\DD$   passes at least one separatrix of $\Fh$ (see
\cite{mol2002}).
Let us start by supposing that  $\DD$ contains only dicritical components with valence up to two. Then
$\hat{F}$ is devoid of poles and $\nu_{0}(\hat{F}) \geq 0$.
Aside from the trivial  case $\DD = D$ with $\val(D) = 0$, for which $\nu_{0}(\hat{F}) = 2$, we
get    $\nu_{0}(\hat{F}) > 1$ in the two following cases:
 \par \smallskip (i)  There exists $D \in \DD$ with $\val(D) = 1$;  here,
  $\hat{F}$ contains a separatrix crossing $D$ and there exists at least one isolated
separatrix crossing the invariant part of $\DD$.
\par \smallskip (ii) There exists $D \in \DD$ with $\val(D) = 2$; then $D$
  disconnects the invariant part of $\DD$ giving rise to two connected components, where we  find at least two isolated separatrices.

Now, we admit the existence of a dicritical component $D \subset \DD$
 of valence $\val(D) \geq 3$. In the course of the reduction process,  $D$ appears   as the result of an explosion at $0 \in \mathbb{C}^{2}$, at
 non-corner point or at a  corner point. Corresponding to these three cases, new sequences explosions will be needed starting  at, respectively,   $\val(D)$, $\val(D) - 1$ or $\val(D) - 2$  points of $D$. Each  of these points  gives rise to at least to one isolated separatrix
 with  algebraic multiplicity no less than that of any dicritical separatrix crossing $D$. Therefore,
 the negative contribution to $\nu_{0}(\hat{F})$ given by the $\val(D) - 2$ dicritical separatrices attached to $D$  will
 be neutralized by the contribution  of these isolated separatrices. Now, in order to see
 that $\nu_{0}(\hat{F}) > 1$, it suffices to take $D$ as the first dicritical component
 with $\val(D) \geq 3$ appearing in the reduction process. The extra separatrices we need are found using the arguments of
  cases (i) and (ii) above.
\end{proof}

\begin{rem}
In the non-dicritical case,   we may have $\nu_{0}^{*}(\hat{F}) = 0$, even if $0 \in \mathbb{C}^{2}$ is a singularity of
$\Fh$. For instance, the local foliation $\Fh$ given by the non-linearizable Poincaré-Dulac normal form
$\omega = (nx + \zeta y^{n}) dy - y dx$, where $\zeta \neq 0$ and $n \in \mathbb{N}$ with $n \geq 2$, has $y = 0$ as the unique separatrix, giving $\nu_{0}(\hat{F}) = 1$. Another example is given by the nilpotent singularity
by $\eta = d(y^2+x^4)+4x^2dy$, whose unique separatrix is $y = - x^{2}$.
\end{rem}

\begin{defn} Let  $\Fh$ be a local foliation at $(\mathbb{C}^{2},0)$ and
$E:(\MM,\DD)\rightarrow(\mathbb{C}^{2},0)$ be a minimal process of reduction of singularities.
If $D \subset \DD$ is an irreducible component, the  \emph{valuation of} $\Fh$
along $D$,
denoted by $\nu_{D}(\Fh)$, is the order of vanishing of $E^{*} \omega$ along $D$, where
$\omega$ is any $1-$form inducing $\Fh$.
\end{defn}

For instance, if $D$ arises from the first blow-up at $0 \in \mathbb{C}^{2}$,
then
\[ \nu_{D}(\Fh) = \begin{cases}
\nu_{0}(\Fh)  \text{ if } D \text{ is
non-dicritical} \smallskip \\
 \nu_{0}(\Fh) + 1 \text{ if } D \text{ is
dicritical}.
\end{cases}
\]

In much the same way, if $\hat{F}$ is a formal meromorphic function at $(\mathbb{C}^{2},0)$,
then, given some process of reduction of singularities $E:(\MM,\DD)\rightarrow(\mathbb{C}^{2},0)$,
we can define the valuation of $\hat{F}$ along $D$, denoted by $\nu_{D}(\hat{F})$,
as the order of vanishing of $E^{*}F = F \circ E$ along $D$.
The following fact also appears  in \cite{genzmer2007}:
\begin{prop}
\label{prop-valuation-2nd-type}
Let $\Fh$ be a foliation of second type having
$\hat{F}$ as a balanced equation of
separatrices.
Let $E:(\MM,\DD)\rightarrow(\mathbb{C}^{2},0)$ be a minimal process of reduction of singularities
for $\Fh$.
Then, for every component $D \in \mathcal{D}$, it holds
\[\nu_{D}(\hat{F}) = \begin{cases}
\nu_{D}(\Fh)+1 \text{ if } D \text{ is
non-dicritical} \smallskip \\
 \nu_{D}(\Fh) \text{ if } D \text{ is
dicritical}.
\end{cases}
\]
In particular, $\nu_{D}(\hat{F}) > 0$ for all $D \in \mathcal{D}$.
\end{prop}

Let  $\Fh$ be a foliation, not necessarily of second type, with balanced equation
of separatrices
$\hat{F}$ and minimal reduction process
$E:(\MM,\DD)\rightarrow(\mathbb{C}^{2},0)$. Following what was done for the algebraic
 multiplicity, we   associate to $\Fh$ an invariant $\nu^{*}_{D}(\Fh)$ called  \emph{pure
 valuation} along $D \in \DD$. As before, this would  be   the
 valuation $\nu_{D}(\Gh)$ for some foliation $\Gh$ of second type sharing with
 $\Fh$ the same reduction structure and the same balanced equation, if such a  $\Gh$
 existed. Since this is not guaranteed, we appeal to
  the following recursive definition:
if $D$ arises from the first blow-up at $0 \in \mathbb{C}^{2}$, then
\[\nu^{*}_{D}(\Fh) = \begin{cases}
\nu_{0}^{*}(\Fh)  \text{ if } D \text{ is
non-dicritical} \smallskip \\
 \nu_{0}^{*}(\Fh) + 1 \text{ if } D \text{ is
dicritical}.
\end{cases}
\]
In order to proceed to the inductive step, write $E = E_{k} \circ \cdots \circ E_{1}$ the divided blow-up,
where   $k \geq 2$, and, for $1 \leq j \leq k$,
   let $D_{j}$ be the divisor associated to $E_{j}$ and $\Fh_j
    = (E_{j} \circ \cdots \circ E_{1})^{*} \Fh$ be the strict transform foliation. Let $1< i < k$ and
suppose that the pure valuation  is defined for irreducible components
  appearing at the height $i$. Suppose that $E_{i+1}$ is a blow-up at a point
  $p \in \DD_{i} = \cup_{j=1}^{i}D_{j}$.  We define:
\begin{equation}
\label{purevaluation}
\nu^{*}_{D_{i+1}}(\Fh) = {\nu}^{*}_{p}( \Fh_{i}) + (1  - \epsilon(D_{i+1})) +
\sum_{D \in V(p)}  \nu^{*}_{D}(\Fh),
\end{equation}
where $V(p)$ stands for the set of irreducible components of $\DD_{i}$ containing $p$ and, for a component $D \subset \DD$,
\begin{equation}
\label{definition-epsilon}
\epsilon (D ) = \begin{cases}
1  \text{ if } D \text{ is
non-dicritical} \smallskip \\
 0 \text{ if } D \text{ is
dicritical}\, .
\end{cases}
\end{equation}
Notice that, as a consequence of Lemma \ref{lemma_pure_multiplicity}, except for the case of a non-dicritical foliation
with a unique smooth separatrix, $\nu^{*}_{D}(\Fh) > 0$ for every $D \in \DD$.
Formula \eqref{purevaluation} is satisfied by the usual valuation. This is the
key for the inductive step in the proof of Proposition~\ref{prop-valuation-2nd-type} in \cite{genzmer2007}.
The same arguments therein give us a new   version for this result:
\begin{prop}
\label{prop-valuation-general}
Let $\Fh$ be a foliation having
$\hat{F}$ as a balanced equation of
separatrices.
Let $E:(\MM,\DD)\rightarrow(\mathbb{C}^{2},0)$ be a minimal process of reduction of singularities
for $\Fh$.
Then, for every component $D \in \mathcal{D}$, it holds
\[\nu_{D}(\hat{F}) = \begin{cases}
\nu_{D}^{*}(\Fh)+1 \text{ if } D \text{ is
non-dicritical} \smallskip \\
 \nu_{D}^{*}(\Fh) \text{ if } D \text{ is
dicritical}.
\end{cases}
\]
In particular, $\nu_{D}(\hat{F}) > 0$ for all $D \in \mathcal{D}$.
\end{prop}

When $\Fh$ is non-dicritical and of  second type, it is well known (see \cite{mattei2004}, Corollary 3.1.10)
that $\Fh$ and $\d\hat{F}$ share the same process of reduction.
However, this fails to be true when $\Fh$ is dicritical. For instance,
the quasi-radial foliation given by $px\d y-qy\d x$, with $p,\ q\in\mathbb{N}^{*}$
relatively prime, admits a balanced equation whose differential is
  $\d\hat{F}=\d\left(xy\right)=x\d y+y\d x$. The latter is
reduced whereas $px\d y-qy\d x$ needs a reduction process attached
to the Euclid's algorithm of the pair $\left(p,q\right)$. Nevertheless,
we can establish  the following link between the two reduction
processes:

\begin{prop}
Let $\Fh$ be a foliation  % of second type
having $\hat{F}$  as a balanced
equation of separatrices. Let $E$ be the reduction process of $\Fh$. Then:
\begin{enumerate}[(a)]
 \item  Any component $D \subset E^{-1}\left(0\right)$  is invariant
by $E^{*}\d\hat{F}$.
 \item  Any singularity of $E^{*}\d\hat{F}$ is reduced except possibly along
the dicritical components of $\Fh$, where $E^{*}\d\hat{F}$ might
have dicritical singularities.
\end{enumerate}
\end{prop}
\begin{proof}
The proof    relies on  Proposition~\ref{prop-valuation-general}.
At a point $p \in D$
there are local coordinates $\left(x,y\right)$  such that the pull-back
of $\hat{F}$ is written in the following way:

\par (1)  if $p$ is neither a zero nor a pole of the strict transform of $\hat{F}$,
then $ (E^{*}\hat{F} )_{p}=x^{\nu_{D} (\hat{F} )}$,
where $x$ is a local equation for $D$ near $p$.

\par (2) if $p$ is a corner, say $p=D_{1}\cap D_{2}$, then $ (E^{*}\hat{F} )_{p}=x^{\nu_{D_{1}} (\hat{F} )}y^{\nu_{D_{2}} (\hat{F})}$,
where $x$ and $y$ are local equations for $D_{1}$ and
$D_{2}$, respectively.

\par (3)  if $p$ is either a zero or a pole of the strict transform of $\hat{F}$,
then $(E^{*}\hat{F})_{p}$ is either $x^{\nu_{D} (\hat{F})}y$
or $x^{\nu_{D}(\hat{F})}/y$, where $x$ is a local
equation for $D$   and $y$ a local equation for the zero or the
pole.

The combination of the above remarks with the upcoming lemma yields the proposition.
\begin{lem}
Let $\hat{H}$ be any meromorphic function, $E$ be any
blowing-up process and $D$ be any component of the exceptional divisor.
Then $D$ is dicritical for $E^{*}\d\hat{H}$ if and only if $\nu_{D}(\hat{H})=0$.
\end{lem}

Indeed, according to Proposition~\ref{prop-valuation-general}, if $D \subset E^{-1}\left(0\right)$ then   $\nu_{D}(\hat{F})$
is strictly positive. Thus, any
component of the exceptional divisor  of $E$ is invariant. Moreover,
  around the points of intersection of the poles of $\hat{F}$ with the divisor,
there are local coordinates $ (x,y )$ such that $ (E^{*}\hat{F} )= x^{N}/y$,
which is a dicritical singularity for $E^{*} \d \hat{F}$, reduced after $N$ blow-ups.
\end{proof}

%In general, the reduction process of $\d\hat{F}$ \emph{dominates}
%that of $\Fh.$ This is not always the case: for instance, for
%the radial foliation $\Fh : x\d y-y\d x =0$, the balanced equation
%$\hat{F}=xy$ is such that  $\d\hat{F}=\d\left(xy\right)=x\d y+y\d x$   does not need any blow-up to
%be reduced whereas $\Fh$ needs one. On the other hand, the   reduction process
%of the alternative balanced equation $\hat{F}=xy(x+y)/(x-y)$ dominates
%that of $\mathcal{F}.$

Even when $\Fh$ is convergent, it is not enough to consider a balanced
equation formed only by convergent invariant curves. Of course,
formal separatrices may appear as weak invariant curves of saddle-nodes.
For instance,
  a balanced equation for the Euler
singularity is given by
the formal equation
\[
\hat{F}=x \left(y-\sum_{n\geq1} (n-1 )!x^{n}\right) .
\]
Evidently, for a convergent foliation $\Fh$, all possible formal separatrices are isolated, the ones in $\textup{Dic}(\Fh)$
being all convergent.

One of the main features of balanced equations is their good behavior
under blow-ups:

\begin{lem}
\label{lemma-balanced}Let $\hat{F}$ be a balanced equation of separatrices for $\Fh$.
Let $\pi: (\widetilde{\mathbb{C}^{2}},D )\to (\mathbb{C}^{2},0 )$
be the standard blow-up at the origin. Then, for any point $p\in D$,
singular for $\pi^{*}\Fh$, the germ at $p$ of the meromorphic function
\[
\frac{\hat{F}\circ\pi}{h_{p}^{\nu_{0} (\hat{F} )-\epsilon (\Fh)}} \ ,
\]
where $h_{p}$ is a local equation of $D$ and $\epsilon(\Fh) = \epsilon(D)$ is defined in \eqref{definition-epsilon},
%\[\epsilon (\Fh ) = \begin{cases}
%1  \text{ if } D \text{ is
%non-dicritical} \smallskip \\
% 0 \text{ if } D \text{ is
%dicritical}\, ,
%\end{cases}
%\]
is a balanced equation
for the germ  of $\pi^{*}\Fh$  at $p$.
\end{lem}
\begin{proof} We examine separately the non-dicritical and the dicritical cases.

 \smallskip
\par \noindent \underline{\it{First case:}} \
%\par \noindent {\it Case 1.}
 $D$ is non-dicritical. We consider the
 reduction of singularities of $\pi^{*}\Fh$ at $p$ as part of that of $\Fh$.
From the point of view of $\pi^{*}\Fh $,  the germ of $D$ at $p$ is no longer part of the divisor,   turning into a separatrix.
The matter is to decide whether  it is an isolated or a dicritical separatrix.
Let $D^{\prime}$ be the component of the desingularization divisor of $\pi^{*}\Fh$  intersecting $D$.
If $D^{\prime}$   is
non-dicritical, then the germ of $D$   is
an isolated separatrix for $\pi^{*}\Fh$. Thus,
\begin{equation}
\label{bal-eq-blowup}
\frac{\hat{F}\circ\pi}{h_{p}^{\nu_{0}(\hat{F})}} \, h_{p}
\end{equation}
%$ (\hat{F}\circ\pi / h_{p}^{\nu_{0} (\hat{F} )}) h_{p}$
is a balanced equation for $ \pi^{*}\Fh$  at $p$. On the other hand, if $D^{\prime}$
is dicritical, then its valence   as a dicritical component of $ \pi^{*}\Fh$ at $p$ is  one unit less
 its valence   as a dicritical component of $\Fh$. Thus, in view of (\ref{eq:1-balance}),
%$\frac{\hat{F}\circ\pi}{h_{p}^{\nu_{0}\left(\hat{F}\right)}}\cdot h_{p}$
equation \eqref{bal-eq-blowup}
is again a balanced equation for $\pi^{*}\Fh$ at $p$.

 \smallskip
\par \noindent \underline{\it{Second case:}} \
%\par \noindent {\it Case 2 .}
$D$ is dicritical.
Then, by the definition of reduction of singularities, the component $D^{\prime}$ touching $D$
in the   reduction of singularities of  $\Fh$ cannot be dicritical. Since
$D$ is not $\pi^{*}\Fh$-invariant, %$\frac{\hat{F}\circ\pi}{h_{p}^{\nu_{0}\left(F\right)}}$
$ \hat{F}\circ\pi / h_{p}^{\nu_{0} (F )}$
is a balanced equation of $\pi^{*}\Fh$ at $p$.
 \end{proof}

This lemma is the key ingredient for most of the properties to be proved
in this article, allowing us to reason inductively on the length
of the  reduction process.

\section{Polar intersection and polar excess}

\label{section-polar-excess}

Let $\hat{\eta}$ be a formal meromorphic $1-$form defined near a
point $p$ of a complex surface. It  can  be regarded  as a $1-$form at $(\mathbb{C}^{2},0)$ by taking analytic
coordinates $ (x,y)$ for which $p = (0,0)$, so that
 \[\hat{\eta} \ =\ \frac{\hat{\omega}}{H} \ = \ \frac{P\d x+Q\d y}{H},\]
where $P,\ Q$ and $H$ are formal functions in $\mathbb{C}\left[\left[x,y\right]\right]$.
For  $(a:b) \in \mathbb{P}^{1}_{\mathbb{C}}$, the \emph{polar curve} of $\hat{\eta}$ with
respect to $(a:b)$ is the curve $\mathcal{P}_{(a:b)}^{\hat{\eta}}$ with formal meromorphic  equation
$(a P + b Q) / H$. In the convergent case, when $a \neq 0$, the points of  $\mathcal{P}_{(a:b)}^{\hat{\eta}}$ outside $(H)_0$ are   those
where $\hat{\eta}$ defines a tangent line with inclination $b/a$.
Now, suppose that  $\hat{B}$ is an irreducible curve invariant by $\hat{\eta}$
such that $\hat{B} \nsubseteq (\hat{\eta})_{\infty}= (H )_{0}$, having
  $\gamma\left(t\right)$ as a formal Puiseux parametrization.
  We   calculate the   intersection number
\[
   (\mathcal{P}_{ (a:b)}^{\hat{\eta}},\hat{B} )_{p} = \textup{ord}_{t=0}\left(\frac{aP+bQ}{H}\circ\gamma\right),
\]
which does not depend on the choice of coordinates $(x,y)$. This enables us to set the following definition:
\begin{defn}
\label{def-polar-intesection}
The \emph{polar intersection number} of  $\hat{\eta}$ and $\hat{B}$ at $p$ is the integer
\[(\mathcal{P}^{\hat{\eta}},\hat{B} )_{p}  = (\mathcal{P}_{ (a:b )}^{\hat{\eta}},\hat{B} )_{p}\]
%\[\mathcal{P}_{p}(\hat{\eta},\hat{B} )  = (\mathcal{P}_{ (a:b )}^{\hat{\eta}},\hat{B} )_{p}\]
obtained  for a generic point  $(a:b) \in \mathbb{P}^{1}_{\mathbb{C}}$.
\end{defn}
This  is an adaptation, for formal meromorphic forms, of the definition in   \cite{cano2015}.
Clearly,  polar intersection numbers are  well defined,
independent  both on the   coordinates and  on  the  choice
of the Puiseux parametrization of $B$.

Polar intersection numbers have a nice behavior under blow-ups:
\begin{lem}
\label{lemma-transform}
Let $\hat{\eta}$ be a formal meromorphic at $p \in \mathbb{C}^{2}$  form  and
$\hat{B}$ be an irreducible invariant curve  as in the previous definition.
Let $\pi: (\widetilde{\mathbb{C}^{2}},D ) \to (\mathbb{C}^{2},p)$ be
 the standard blow-up at $p$. Denote by $\tilde{B}$
the strict transform of $\hat{B}$ and $q=\tilde{B}\cap D$. Then
\[
 (\mathcal{P}^{\pi^{*}\hat{\eta}},\tilde{B})_{q} = (\mathcal{P}^{\hat{\eta}},\hat{B} )_{p}+\nu_{q} (\tilde{B}).
\]
\end{lem}
\begin{proof}
Let us fix coordinates $(x,y)$ and a Puiseux
parametrization $\gamma (t)= (x(t),y (t) )$ of $\hat{B}$
with $\nu_{p}(\hat{B})=\ord_{t=0}x(t)$. Let us write
$\pi (x,u)=(x,ux)$. In these coordinates, $\tilde{B}$
is parametrized by $\tilde{\gamma}(t)=(x(t),u(t)=y(t)/x(t))$.
We have
\[
\pi^{*}\hat{\eta}=\frac{1}{H(x,ux)}\left((P(x,ux)+uQ(x,ux))\d x+xQ(x,ux)\d u\right).
\]
For $(a:b) \in \mathbb{P}^{1}_{\mathbb{C}}$, we   make the following computation
\begin{eqnarray*}
(\mathcal{P}_{(a:b)}^{\pi^{*}\hat{\eta}},\tilde{B})_{q} & = & \textup{ord}_{t=0}\left(\frac{a\left(P\left(x,ux\right)+uQ\left(x,ux\right)\right)+bxQ\left(x,ux\right)}{H\left(x,ux\right)}\circ\tilde{\gamma}\right)\\
 & = & \ord_{t=0}\left(\frac{a\left(P\circ\gamma\left(t\right)+u\left(t\right)Q\circ\gamma\left(t\right)\right)+bx\left(t\right)Q\circ\gamma\left(t\right)}{H\circ\gamma\left(t\right)}\right)\\
 & = & \min\left\{\ord_{t=0}(P\circ\gamma\left(t\right)+u\left(t\right)Q\circ\gamma\left(t\right)),\ord_{t=0}(x\left(t\right)Q\circ\gamma\left(t\right))\right\}\\
 &  & -\ord_{t=0}H\circ\gamma\left(t\right).
\end{eqnarray*}
Now, since $\hat{B}$ is invariant, we have $P\circ\gamma(t) x^{\prime}\left(t\right)+Q\circ\gamma(t) y^{\prime}\left(t\right)=0$.
Noting that $\ord_{t=0}x(t) \leq \ord_{t=0}y(t)$,     this gives in particular that  $\ord_{t=0} Q\circ\gamma(t) \leq \ord_{t=0} P\circ\gamma(t)$ and therefore
$(\mathcal{P}_{(a:b)}^{\hat{\eta}},\hat{B})_{p} = \ord_{t=0} Q\circ\gamma(t) - \ord_{t=0} H\circ\gamma(t)$.
It is straightforward that
\begin{eqnarray*}
\ord_{t=0}(P\circ\gamma\left(t\right)+u\left(t\right)Q\circ\gamma\left(t\right)) & = & \ord_{t=0}\left(\frac{-u^{\prime}\left(t\right)x\left(t\right)}{x^{\prime}\left(t\right)}Q\circ\gamma\left(t\right)\right)\\
 & = & \ord_{t=0}u\left(t\right)+\ord_{t=0}Q\circ\gamma\left(t\right).
\end{eqnarray*}
We finally find
\begin{eqnarray*}
 (\mathcal{P}_{(a:b)}^{\pi^{*}\hat{\eta}},\tilde{B})_{q} & = & \min\left\{\ord_{t=0}u\left(t\right)+\ord_{t=0}Q\circ\gamma\left(t\right),\ord_{t=0}x\left(t\right)+\ord_{t=0}Q\circ\gamma\left(t\right)\right\}\\
 &  & -\ord_{t=0}H\circ\gamma\left(t\right)\\
 & = & \min\left\{\ord_{t=0}u\left(t\right),\ord_{t=0}x\left(t\right)\right\}+
 \ord_{t=0}Q\circ\gamma\left(t\right)-\ord_{t=0}H\circ\gamma\left(t\right)\\
 & = & \nu_{q}(\tilde{B})+(\mathcal{P}_{(a:b)}^{\hat{\eta}},\hat{B})_{p}.
\end{eqnarray*}
The proof is finished by taking generic $(a:b) \in \mathbb{P}^{1}_{\mathbb{C}}$.
\end{proof}

Our interest does not lie in the absolute values of   polar intersection numbers.  The idea is
to compare polar intersection numbers of a foliation and
 a ``reference foliation'' having the balanced equation as a first integral. This is the same
 principle developed in \cite{cano2015} for the non-dicritical case. There, however, the finiteness of the set of separatrices gives a straight   choice for this reference
 foliation.    More specifically, we define the
following invariant:

\begin{defn}
\label{def-polar-variation}
Let $\Fh$ be a germ of singular foliation at a point $p \in \mathbb{C}^{2}$ having
 $\hat{F}$ as a balanced equation of separatrices. Let $C \subset  (\hat{F})_{0}$ be
a union of zeros of $\hat{F}$ and consider the decomposition in irreducible
components $C=\cup_{i=1}^{n}C_{i}$. We define the \emph{polar excess index} of $\Fh$
with respect to $C$ as
\[
\var_{p} (\Fh,C) = \sum_{i=1}^{n}\var_{p}(\Fh,C_{i})=
\sum_{i=1}^{n} \left((\mathcal{P}^{\Fh},C_{i})_{p}-(\mathcal{P}^{\d\hat{F}},C_{i})_{p} \right),
\]
where $(\mathcal{P}^{\d\hat{F}},C_{i})_{p}$ refer to the polar intersection numbers of the formal foliation defined
by $\d\hat{F}$.
We also introduce a relative version of the polar  excess index:
if $f$ is the formal equation of the curve of separatrices
$C \subset  (\hat{F})_{0}$ above, instead of using the whole balanced equation $\hat{F}=G / H$ as done
 in the calculation of the polar excess, we take $f/H$. More precisely:
\[
\var_{p}^{\textup{rel}}(\Fh,C)=\sum_{i=1}^{n} \left((\mathcal{P}^{\Fh},C_{i})_{p}-(\mathcal{P}^{\d(f/H)},C_{i})_{p} \right).
\]
\end{defn}

\begin{example} \label{example-reduced}
{\rm Let us see what happens to $\var_{p}(\Fh,B)$ for a branch of separatrix $B$
 in the reduced case.
\par   (1) $\Fh$ is non singular at $p$. In this case, $B$ is the local leaf
at $p$ and it is easy to see that $\var_{p}(\Fh,B) = 0$.
\par   (2) $\Fh$ has a reduced non-degenerate singularity at $p$.
Now we can  take local coordinates such
that $p = (0,0)$ and $\Fh$ is given by
\[
x\left(1+u\left(x,y\right)\right)\d y+y\left(\lambda+v\left(x,y\right)\right)\d y.
\]
The function $\hat{F}=xy$ is a balanced equation.
If $B$ is either $x=0$ or $y=0$, it is  straight that
$(\mathcal{P}^{\Fh},B)_{p}  = 1$
and also $(\mathcal{P}^{\d\hat{F}},B)_{p} = 1$. Therefore   $\var_{p}(\Fh,B) = 0$.
\par   (3) $\Fh$ has a saddle-node at $p$.
We take the formal coordinates $(x,y)$ inducing  the normal form  of equation \eqref{eq-saddle-node-form}:
\[\left(\zeta x^{k}-k\right)y\d x+x^{k+1}\d y,\quad k\in\mathbb{N}^{*},\zeta\in\mathbb{C} .\]
Again, $\hat{F}=xy$ is a balanced equation. If $B$ is the strong separatrix, $x=0$, we find
$(\mathcal{P}^{\Fh},B)_{p}   = 1 $
and, thus, $\var_{p}(\Fh,B) = 0$. On the other hand, if $B$ is the weak separatrix, $y=0$, we have
$(\mathcal{P}^{\Fh},B)_{p} =   k+1$,
that results in $\var_{p}(\Fh,B) = k > 0$.
}
\end{example}

Keeping the notation of Lemma~\ref{lemma-transform}, we have:
\begin{prop}
\label{prop:Var1} Let $B$ be an irreducible component of $ (\hat{F})_{0}.$
Then
\[
\var_{p}(\Fh,B)=\var_{q}(\pi^{*}\Fh,\tilde{B})+\tau(\Fh)\nu_{p}(B) ,
\]
where $\tau(\Fh)$ is  tangency excess of $\Fh$ defined in equation \eqref{eq-tau}.
In particular,
\[ \var_{p}(\Fh,B) \geq \var_{q}(\pi^{*}\Fh,\tilde{B}),\]
equality
holding if and only if $\Fh$ is of second type.
\end{prop}
\begin{proof}
Consider   adapted coordinates $(x,u )$ for which $\pi (x,u)= (x,ux)$
and $B$ is given by $\gamma (t)=(x(t),y(t))$,
where $\nu_{p}(B)=\ord_{t=0}x(t)$. In order
to avoid  any confusion, let us denote by $\hat{F}_{p}$  the balanced
equation for $\Fh$ at $p$ and $\hat{F}_{q}$ the balanced equation for $\pi^{*}\Fh$ at $q$.
Following Lemma~\ref{lemma-balanced}, the relation between the two balanced
equations is
\[
\pi^{*}\hat{F}_{p}=x^{\nu_{p}(\hat{F}_{p})-\epsilon(\Fh)}\hat{F}_{q}.
\]
Taking derivatives in the above relation, we get
\[
\pi^{*}\d\hat{F}_{p}=
x^{\nu_{p}(\hat{F}_{p})-\epsilon(\Fh)}
\d\hat{F}_{q}+(\nu_{p}(\hat{F}_{p})-
\epsilon(\Fh))x^{\nu_{p}(\hat{F}_{p})-\epsilon(\Fh)-1} \hat{F}_{q}\d x.
\]
Since $\hat{F}_{q}$ vanishes along $\tilde{B}$, for fixed $(a:b) \in \mathbb{P}^{1}_{\mathbb{C}}$ one has
\[
 (\mathcal{P}_{(a:b)}^{\pi^{*}\d\hat{F}_{p}},\tilde{B})_{q}=\ord_{t=0}x(t)^{\nu_{p}(\hat{F}_{p})-\epsilon(\Fh)}+
 (\mathcal{P}_{(a:b)}^{\d\hat{F}_{q}},\tilde{B})_{q}.
\]
Now, taking $(a:b) \in \mathbb{P}^{1}_{\mathbb{C}}$ generic and   using Lemma~\ref{lemma-transform}, we obtain
\begin{equation}
(\mathcal{P}^{\d\hat{F}_{q}},\tilde{B})_{q}=(\mathcal{P}^{\d\hat{F}_{p}},B)_{p}-(\nu_{p}(\hat{F}_{p})-\epsilon(\Fh))
\nu_{p}(B)+\nu_{q}(\tilde{B}). \label{eq:3}
\end{equation}
Moreover, following \cite[Proposition 3]{cano2015} we also have
\begin{equation}
(\mathcal{P}^{\pi^{*}\Fh},\tilde{B})_{q}=(\mathcal{P}^{\Fh},B)_{p}-(\nu_{p}(\Fh)+1-\epsilon(\Fh))\nu_{p}
(B)+\nu_{q}(\tilde{B}).\label{eq:4}
\end{equation}
Combining \eqref{eq:3}, \eqref{eq:4} and Proposition~\ref{prop:Equa-Ba} yields
\begin{eqnarray*}
\var_{q}(\pi^{*}\Fh,\tilde{B}) & = & \var_{p}(\Fh,B)-(\nu_{p}(\Fh)+1-\nu_{p}(\hat{F}_{p}))\nu_{p}(B)\\
 & = & \var_{p}(\Fh,B)-\tau(\Fh)\nu_{p}(B).
\end{eqnarray*}
The final statement follows from $\tau(\Fh) \geq 0$, this number vanishing
if and only if $\Fh$ is a foliation of second type.
\end{proof}

\begin{prop}
\label{prop:VarPos}
Let $\Fh$ be a germ of formal foliation at $p \in S$ having $\hat{F}$ as a balanced equation of separatrices. If $B$ is   an irreducible component of $(\hat{F})_{0}$, then
\[
\var_{p}(\Fh,B)\geq0.
\]
\end{prop}
\begin{proof}
The proof goes by  induction on the length of the reduction process of $\Fh$.
The inductive step is an obvious consequence of Proposition~\ref{prop:Var1}.
The initialization in its turn follows from the three cases in
Example~\ref{example-reduced}.
\end{proof}

As a consequence of the above, we obtain that the polar excess   somehow
gives a measure the existence of saddle-nodes in the desingularization of a foliation. This is the content of
\begin{maintheorem}
\label{thm-generalized-curve}
Let $\Fh$ a germ of singular foliation at $p$ in a surface $S$ and $\hat{F}$
a   balanced equation for its    separatrices. Then $\Fh$
is a generalized curve if and only if
\[
\var_{p}(\Fh,(\hat{F})_{0})=0.
\]
\end{maintheorem}
\begin{proof}
Suppose first that  $\Fh$ is a generalized curve. Then $\Fh$ is in particular a foliation of second type.
For a fixed branch $B \subset (\hat{F})_{0}$,
Proposition~\ref{prop:Var1} asserts that the polar excess $\var$ is invariant under blow-ups. Thus, it suffices to follow the transforms of $B$
along the  the reduction of singularities of $\Fh$.
Now, from Example~\ref{example-reduced},
$\var_{p}(\Fh,B) > 0$  if and only if $B$ is transformed into the weak separatrix of a saddle-node.
We are however in the generalized curve case and we must have $\var_{p}(\Fh,B) = 0$.
We finally conclude that $\var_{p}(\Fh,(\hat{F})_{0})=0$ by summing up $\var_{p}(\Fh,B)$ for all branches $B \subset (\hat{F})_{0}$.

Reciprocally, note that $\var_{p}(\Fh,(\hat{F})_{0})=0$ implies that
$\var_{p}(\Fh,B)=0$ for every branch $B \subset (\hat{F})_{0}$.
Since $(\hat{F})_{0}$ is non-empty, Proposition~\ref{prop:Var1} ensures that $\mathcal{\tau(\Fh)}=0$ and
thus $\Fh$ is  a foliation of second type. This means that all possible saddle-nodes  in the reduction of $\Fh$
are well-oriented.  On the other hand, the weak separatrix of a well-oriented saddle-node
would  contribute positively to the polar  excess of $\Fh$,
  as shown in Example~\ref{example-reduced}. This leads to the conclusion that saddle-nodes do not exist at all and, by definition, $\Fh$
is a generalized curve at $p$.
\end{proof}

\section{Polar excess and the $GSV$-index}

\label{section-GSVindex}

 The  $GSV$-index was defined
by X. Gómez-Mont, J. Seade and A. Verjovsky in \cite{gomezmont1991}
for a holomorphic vector field on an analytic hypersurface $V$ at $(\mathbb{C}^{n},0)$ having isolated singularity. It is
 the  Poincaré-Hopf index of a differentiable vector field obtained by isotopically displacing the original vector field over the Milnor fiber of $V$.
The formulation below, for  foliations on $(\mathbb{C}^{2},0)$ having an invariant curve, was  introduced by
M. Brunella in \cite{brunella1997II}.
\begin{defn}
Let $\mathcal{F}$ be an analytic foliation at $\left(\mathbb{C}^{2},0\right)$
and $C$ be the union of some analytic separatrices of $\mathcal{F}$. If $\omega$
is a $1$-form that induces $\FF$ and $f=0$ is a reduced equation
for $C$, then it is possible to write a decomposition
\[
g\omega=k\d f+f\eta,
\]
where $\eta$ is a $1$-form and $g,~k\in\mathbb{C}\left\{ x,y\right\} $
with $g$ and $f$ relatively prime. The $GSV$-index of $\FF$ with respect to $C$
at $\left(\mathbb{C}^{2},0\right)$ is defined by
\[
GSV_{0}^{\textup{c}}\left(\mathcal{F},C\right)=\frac{1}{2\pi i}\int_{\partial C}\frac{g}{k}\d\left(\frac{k}{g}\right).
\]
Here $\partial C$ is the intersection $C\cap S_{\epsilon}^{3}$ , where $S_{\epsilon}^{3}$
is a small sphere centered at $0\in\mathbb{C}^{2}$, oriented as the
boundary of $C\cap B_{\epsilon}^{4}$ , for a ball $B_{\epsilon}^{4}$
such that $S_{\epsilon}^{3}=\partial B_{\epsilon}^{4}$.
\end{defn}
The superscript ``c'' in $GSV_{0}^{\textup{c}}\left(\mathcal{F},C\right)$ stresses the fact that
this definition works in the convergent category. Nevertheless,
we  can   extended the $GSV$-index to formal foliations using
the following:

\begin{defn}
Let $\Fh$ be a formal foliation at $ (\mathbb{C}^{2},0)$
and $\hat{C}$ be an irreducible separatrix of $\mathcal{F}$. If
$\hat{\omega}$ is a $1$-form   inducing $\Fh$ and $f=0$ is a
reduced equation for $\hat{C}$, then, as in the convergent case, it is possible to write a decomposition
\[
g\hat{\omega}=k\d f+f\hat{\eta}
\]
where $\hat{\eta}$ is a formal $1$-form and $g,~k\in\mathbb{C}\left[\left[x,y\right]\right]$
with $g$ and $f$ relatively prime. Now, if $\gamma$ is a Puiseux parametrization
of $\hat{C},$ we define
\[
GSV_{0}(\Fh,\hat{C})=\ord_{t=0}\frac{k}{g}\circ\gamma.
\]
If $\hat{C} = \hat{C_{0}}\cup\hat{C_{1}}$ is the union of two disjoint sets of separatrices, then we define
the $GSV$-index inductively by the formula
\begin{equation}
GSV_{0}(\Fh,\hat{C})=GSV_{0}(\Fh,\hat{C}_{0})+GSV_{0}(\Fh,\hat{C}_{1})-2(\hat{C_{0}},\hat{C_{1}})_{0} \label{eq:5},
\end{equation}
where $(\hat{C_{0}},\hat{C_{1}})_{0}$ stands for the intersection number at $0 \in \mathbb{C}^{2}$.

\end{defn}
The relation (\ref{eq:5})  simply follows
the relation satisfied in the convergent category as shown in \cite{brunella1997II}. Thus,
the following lemma is straightforward and justifies \emph{a posteriori
}the definition.
\begin{lem}
Let $\mathcal{F}$ be a germ of analytic foliation and let $C$
be a union of convergent separatrices. Then
\[
GSV_{0}\left(\mathcal{F},C\right)=GSV_{0}^{\textup{c}}\left(\mathcal{F},C\right).
\]
\end{lem}

Next, we establish  a link between  the $GSV$-index and the relative polar excess.
\begin{prop}
\label{prop:GSV} Let $\Fh$ be a germ of singular foliation at $p \in \mathbb{C}^{2}$
 and $\hat{F}$ be a balanced equation of separatrices. If $B \subset (\hat{F})_{0}$ is
irreducible then
\[
\var_{p}^{\textup{rel}} (\Fh,B )=GSV_{p} (\Fh,B )+ ( (\hat{F} )_{\infty},B)_{p} .
\]
\end{prop}
\begin{proof}
Let us write $\omega=P\d x+Q\d y$, $\hat{F}=G / H$ and consider
$f$ a reduced equation of $B.$ The decomposition $g\omega = k\d f + f\eta$ gives us
\[
aP+bQ=\frac{k}{g}\frac{\left(a\partial_{x}f+b\partial_{y}f\right)}{H}H+\frac{f}{g}\left(a\eta_{x}+b\eta_{y}\right),
\]
where $\eta = \eta_{x} \d x + \eta_{y} \d y$.
Let   $\gamma$ be a Puiseux parametrization of $B$. Since
$f$ and $g$ are relatively prime, $g\circ\gamma$ is not identically zero. Moreover,
by definition, $f\circ\gamma=0$. Hence, for a fixed $(a:b) \in \mathbb{P}^{1}_{\mathbb{C}}$, we obtain
\[
\underbrace{\ord_{t=0}\left(aP+bQ\right)\circ\gamma}_{ (\mathcal{P}_{\left(a:b\right)}^{\Fh},B)}=
\underbrace{\ord_{t=0}\frac{k}{g}\circ\gamma}_{GSV_{p}(\Fh,B)}+
\underbrace{\ord_{t=0}
\frac{\left(a\partial_{x}f+b\partial_{y}f\right)}{H}}_{\left(\mathcal{P}_{\left(a:b\right)}^{\d\left(f/H\right)},B\right)_{p}}+
\underbrace{\ord_{t=0}H\circ\gamma}_{((\hat{F})_{\infty},B)_{p}}.
\]
which, after taking generic $(a:b)$, gives    the proposition. \end{proof}

\begin{lem}
\label{lem:adjunc}Let $B_{1}$ and $B_{2}$ be two  branches
of $(\hat{F})_{0}$. Then
\[
\var_{p}^{\textup{rel}}(\Fh,B_{1}\cup B_{2})=\var_{p}^{\textup{rel}}(\Fh,B_{1})+\var_{p}^{\textup{rel}}(\Fh,B_{2})-2(B_{1},B_{2})_{p}.
\]
\end{lem}
\begin{proof}
Let $f$ and $g$ be respectively  the equations of $B_{1}$ and $B_{2}.$
Let $\gamma$ be a Puiseux parametrization of $B_{1}.$ For $(a:b) \in \mathbb{P}^{1}_{\mathbb{C}}$, we have
\begin{eqnarray*}
 (\mathcal{P}_{\left(a:b\right)}^{d ( fg/H)},B_{1} )_{p} & = & \ord_{t=0}\left(\frac{a\partial_{x}(fg)+b\partial_{y}(fg)}{H}\right)\circ\gamma\\
 & = & \ord_{t=0}\left(\frac{g\left(a\partial_{x}f+b\partial_{y}f\right)+f\left(a\partial_{x}g+b\partial_{y}g\right)}{H}\right)\circ\gamma\\
 & = & \ord_{t=0}\left(g\circ\gamma\right)+\ord_{t=0}\left(\frac{a\partial_{x}f+b\partial_{y}f}{H}\right)\circ\gamma\\
 & = &  (B_{1},B_{2})_{p}+ (\mathcal{P}_{ (a:b)}^{d(f/H)},B_{1})_{p}.
\end{eqnarray*}
Since, by symmetry,  the same holds for $B_{2}$, we get
\begin{eqnarray*}
\var_{p}^{\textup{rel}}(\Fh,B_{1}\cup B_{2}) & = & (\mathcal{P}_{\left(a:b\right)}^{\Fh},B_{1}\cup B_{2})_{p}-(\mathcal{P}_{(a:b)}^{\d(fg/H)},B_{1}\cup B_{2})_{p}\\
 & = & (\mathcal{P}_{(a:b)}^{\Fh},B_{1})_{p}+(\mathcal{P}_{(a:b)}^{\Fh},B_{2})_{p}\\
 &  & -(\mathcal{P}_{(a:b)}^{\d(fg/H)},B_{1})_{p}-(\mathcal{P}_{(a:b)}^{\d(fg/H)},B_{2})_{p}\\
 & = & (\mathcal{P}_{(a:b)}^{\Fh},B_{1})_{p}-(\mathcal{P}_{(a:b)}^{d(\frac{f}{H})},B_{1})_{p}\\
 &  & +(\mathcal{P}_{(a:b)}^{\Fh},B_{2})_{p}-(\mathcal{P}_{(a:b)}^{d(g/H)},B_{2})_{p}  \\
  &   &  -2(B_{1},B_{1})_{p} \\
 &  =  &  \var_{p}^{\textup{rel}}(\Fh,B_{1})+\var_{p}^{\textup{rel}}(\Fh,B_{2})-2(B_{1},B_{2})_{p}
 \end{eqnarray*}

\end{proof}
By   simple induction on the number of irreducible components, we can
extend   Proposition~\ref{prop:GSV},  replacing $B_{1}$ and $B_{2}$ by
 two disjoint sets of separatrices. Moreover, since the $GSV$-index
satisfies the very same adjunction formula,
this
result   implies that
\begin{cor}
\label{cor:rel-gsv}
For any set of separatrices $C\subset (\hat{F} )_{0}$,
one has
\[
\var_{p}^{\textup{rel}}(\Fh,C)=GSV_{p}(\Fh,C)+((F)_{\infty},C)_{p}.
\]
\end{cor}
In order to make a link between the total polar excess of $\Fh$ and its
$GSV$-index, we will start by providing a connection between the total and the relative polar
excess.
\begin{lem}
\label{lem:rel-norel} Let $C\subset(\hat{F})_{0}$ be a set of separatrices. Then
\[
\var_{p}(\Fh,C)=\var_{p}^{\textup{rel}}(\Fh,C)-(C,(\hat{F})_{0}\backslash C)_{p}.
\]
\end{lem}
\begin{proof}
We will present the proof for a  branch of separatrix. The general case
is a simple induction based upon Lemma~\ref{lem:adjunc}. As usual,
denote  $\hat{F}= G/H =  g_{1}\cdots g_{n}/H$ and let $\gamma$ be a
Puiseux parametrization of the irreducible component $B=\{ g_{1}=0\} $
of $(\hat{F})_{0}$ . Let us denote by $g$ the product
$g_{2}\cdots g_{n}$. Then, for $(a:b) \in \mathbb{P}^{1}_{\mathbb{C}}$, we have
\begin{eqnarray*}
(\mathcal{P}_{(a:b)}^{d(\frac{G}{H})},B)_{p} & = & \ord_{t=0}\left(\frac{a\partial_{x}G+b\partial_{y}G}{H}\right)\circ\gamma\\
 & = & \ord_{t=0}\left(g\left(\frac{a\partial_{x}g_{1}+b\partial_{y}g_{1}}{H}\right)\right)\circ\gamma\\
 & = & \underbrace{\ord_{t=0}g\circ\gamma}_{(B,(F)_{0}\backslash B)}+\underbrace{\ord_{t=0}\left(\frac{a\partial_{x}g_{1}+b\partial_{y}g_{1}}{H}\right)
 \circ\gamma}_{\left(\mathcal{P}_{(a:b)}^{d(g_{1}/H)},B\right)_{p}} \  ,
\end{eqnarray*}
which ensures the lemma.
\end{proof}
Finally, the combination of Corollary~\ref{cor:rel-gsv} and Lemma~\ref{lem:rel-norel}
yields the next result, which expresses the $GSV$-index in terms of the polar excess index
and the balanced equation of separatrices.

\begin{maintheorem}
\label{thm-gsv-polar}
Let $\Fh$ be a germ of singular foliation at $(\mathbb{C}^{2},p)$.
Let $\hat{F}$ be a balanced equation of separatrices and $C$ be a
subset of $(\hat{F})_{0}$. Then
\[
GSV_{p}(\Fh,C)=\var_{p}(\Fh,C)+(C,(\hat{F})_{0}\backslash C)_{p}-(C,(\hat{F})_{\infty})_{p}.
\]
In particular, if $\Fh$ is a generalized curve then
\[
GSV_{p}(\Fh,(\hat{F})_{0})=-((\hat{F})_{0},(\hat{F})_{\infty})_{p}.
\]
\end{maintheorem}
In the non-dicritical case, the $GSV$-index and the polar excess index coincide,
as shown in \cite{cano2015}:
\begin{cor} If $\Fh$ is non-dicritical   and $C$ is its complete set of separatrices, then
\[ GSV_{p}(\Fh,C)=\var_{p}(\Fh,C). \]
 \end{cor}

\section{The Poincaré problem for dicritical singularities}

\label{section-poincaré-problem}

We now work with global foliations on  $\mathbb{P}^{2}_{\mathbb{C}}$. Let $\mathcal{F}$ be an analytic foliation   of degree
$d$ having an  invariant algebraic curve $S$ of degree $d_{0}$.
In \cite{brunella1997I} and \cite{brunella1997II}, the sum of the $GSV$-indices over $S$ is expressed as
\[c_1(N_\mathcal{F}) \cdot S - S \cdot S = \sum_{p \in \sing(\mathcal{F}) \cap S} GSV_p(\mathcal{F},S), \]where
$N_\mathcal{F} = \mathcal{O}(d + 2)$ is the normal bundle of $\mathcal{F}$.
This translates into the following numerical identity
\[
\left(d+2-d_{0}\right)d_{0}=\sum_{p \in \sing(\mathcal{F})\cap S}GSV_{p}\left(\mathcal{F},S\right).
\]

Using the expression of the $GSV$-index in terms of polar excess given
in Theorem~\ref{thm-gsv-polar}, we get a control of the degree of the
invariant curve  $d_{0}$ in terms of local   invariants of the foliation. We have:

\begin{eqnarray}
d_{0} & = & d+2-\frac{1}{d_{0}}\sum_{p\in \sing(\mathcal{F})\cap S}
GSV_{p}(\mathcal{F},S)\nonumber \\
 & = & d+2-\frac{1}{d_{0}}\sum_{p\in \sing(\mathcal{F})\cap S} \left(\var_{p}(\mathcal{F},S)+(S,(\hat{F}_{p})_{0}\backslash S)_{p}-(S,(\hat{F}_{p})_{\infty})_{p} \right) , \nonumber
\end{eqnarray}
where, at each point $p \in \sing(\mathcal{F})\cap S$, we chose a balanced equation $\hat{F}_{p}$ adapted to
to the local branches of $S$ at $p$.
Using Proposition~\ref{prop:VarPos}, we find
\begin{maintheorem}
\label{thm-poincare-dic}
 Let $\mathcal{F}$ be an analytic foliation on  $\mathbb{P}^{2}_{\mathbb{C}}$   of degree
$d$ having an  invariant algebraic curve $S$ of degree $d_{0}$. Then
\begin{equation} d_{0} \leq
d+2+\frac{1}{d_{0}}\sum_{p\in \sing(\mathcal{F})\cap S} \left[
(S, (\hat{F}_{p})_{\infty})_{p}-(S,(\hat{F}_{p})_{0}\backslash S)_{p} \right], \label{eq:3-1} \end{equation}
where $\hat{F}_{p}$ is a balanced equation of separatrices adapted to
to the local branches of $S$ at $p$.
Besides,   if all singularities
of $\mathcal{F}$ over $S$  are generalized curves, then the equality holds.
\end{maintheorem}

This theorem allows us to recover the main result in \cite{carnicer1994}:

\begin{cor} If $(\hat{F}_{p})_{\infty} = \emptyset$ for all $p \in \sing(\mathcal{F})\cap S$ then
 $d_{0} \leq d+2$. This happens in particular when all such points are non-dicritical.
\end{cor}

%\begin{equation}
%d_{0}^{2}=d_{0}\left(d+2\right)+\sum_{s\in\textup{Sing}\left(\mathcal{F}\right)\cap S}\left(S,\left(\hat{F}_{s}\right)_{\infty}\right)-\left(S,\left(\hat{F}_{s}\right)_{0}\backslash S\right)\label{eq:4-1}
%\end{equation}

\subsection{A dicritical foliation of degree $1$ in $\mathbb{P}^{2}_{\mathbb{C}}$}
\label{subsection-node}

As an example, let us consider the foliation of $\mathbb{P}^{2}_{\mathbb{C}}$
given by the meromorphic homogeneous first integral
\[
\omega=\d\left(\frac{x^{p}z^{q-p}}{y^{q}}\right),
\]
where $p<q$. This is a foliation of degree $1$ on $\mathbb{P}^{2}_{\mathbb{C}}$ whose generic
separatrix has degree $q$ and is given by $x^{p}y^{q-p}-cz^{q}=0$. Let
us consider $S=\left\{ x^{p}y^{q-p}-z^{q}=0\right\} $. The foliation
has three singularities: two of them --- denoted by $a$ and $b$ on
the picture below --- are dicritical, the other one is a saddle.

\begin{figure}[H]
\noindent \centering{}\includegraphics[scale=0.5]{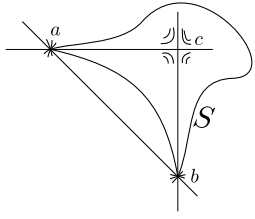}\protect\caption{The foliation of $\mathbb{P}^{2}_{\mathbb{C}}$ given by $\omega= \d ( x^{p}z^{q-p}/y^{q})$}
%$\omega=\protect\d\left(\frac{x^{p}z^{q-p}}{y^{q}}\right)$}
\end{figure}

In this situation, the sum in inequality (\ref{eq:3-1}) reduces to
\[
(S,(F_{a})_{\infty})-(S,(F_{a})_{0}\backslash S)+(S,(F_{b})_{\infty})-(S,(F_{b})_{0}\backslash S).
\]

In a neighborhood of $a$, in the affine chart $z=1$, we can
choose
\[ \hat{F}_{a}= \frac{xy (x^{p}-y^{q} )}{x^{p}-2y^{q}}\]
 as a
 balanced equation. Thus
\begin{eqnarray*}
(S,(F_{a})_{\infty})-(S,(F_{a})_{0}\backslash S) & = & (x^{p}-y^{q},x^{p}-2y^{q})-(x^{p}-y^{q},xy)\\
 & = & pq-p-q.
\end{eqnarray*}

On the other hand,
in a neighborhood of $b$ and in the affine chart $x=1$,   $\hat{F}_{b}= zy (z^{q-p}-y^{q})/(z^{q-p}-2y^{q})$
is a balanced equation of separatrices. Thus
\begin{eqnarray*}
(S,(F_{b})_{\infty})-(S,(F_{b})_{0}\backslash S) & = & (z^{q-p}-y^{q},z^{q-p}-2y^{q})-
(z^{q-p}-y^{q},zy)\\
 & = & q^{2}-qp-2q+p.
\end{eqnarray*}
Finally, we get   checked the  equality
\[
\underbrace{d_{0}}_{=q}=\underbrace{d}_{=1}+2-\frac{1}{q}((pq-p-q)+(q^{2}-qp-2q+p)).
\]

\subsection{The pencil of Lins-Neto.}
\label{subsection-linsneto}

One of the most remarkable counterexamples for the Poincaré's problem
is the pencil of A. Lins-Neto (see \cite{linsneto2002}) defined
by
\begin{eqnarray*}
\mathcal{P}_{\alpha} & = & \omega+\alpha\eta\\
\omega & = & \left(x^{3}-1\right)x\d y+\left(y^{3}-1\right)y\d x\\
\eta & = & \left(x^{3}-1\right)y^{2}\d y+\left(y^{3}-1\right)x^{2}\d x.
\end{eqnarray*}
When $\alpha\in\mathbb{Q}\left(j\right)$ where $j=e^{\frac{2i\pi}{3}}$,
the foliation $\mathcal{P}_{\alpha}$ of $\mathbb{P}^{2}_{\mathbb{C}}$ of degree
$4$ admits a meromorphic first integral whose degree cannot be bounded,
although the singular locus and the local analytical type of the singularities
do not depend on $\alpha$. Actually, when $\alpha\in\mathbb{Q}\left(j\right)\backslash\left\{ 1,j,j^{2},\infty\right\} $,
the singular locus consists of 21 points: 9 of them are non-degenerate
singularities and the 12 remaining are radial singularities, given in local coordinates $(u,v)$
  by $ \d(u/v)=0$. These singularities are the locus of
intersection of the 9 lines defined in homogeneous coordinates by
\[
\left(x^{3}-y^{3}\right)\left(x^{3}-z^{3}\right)\left(y^{3}-z^{3}\right)=0.
\]
The group of automorphisms of $\mathcal{P}_{\alpha}$ is generated
by the following five transformations:
\[
\begin{array}{c}
\left[x:y:z\right]\mapsto\left[y:x:z\right],\ \left[x:y:z\right]\mapsto\left[z:y:x\right],\ \left[x:y:z\right]\mapsto\left[x:z:y\right], \smallskip \\
\left[x:y:z\right]\mapsto\left[jx:j^{2}y:z\right],\ \left[x:y:z\right] \mapsto\left[j^{2}x:jy:x\right] .
\end{array}
\]
The action of this group on the radial singular points divides them
 in $4$ classes of $3$ points:
 \[
\begin{array}{l}
\left\{ s_{1}=\left[0:0:1\right],\left[0:1:0\right],\left[1:0:0\right]\right\}, \smallskip  \\
\left\{ s_{2}=\left[1:1:1\right],\left[j:j^{2}:1\right],\left[j^{2}:j:1\right]\right\}, \smallskip \\
\left\{ s_{3}=\left[1:j:1\right],\left[j:1:1\right],\left[j^{2}:j^{2}:1\right]\right\}, \smallskip  \\
\left\{ s_{4}=\left[1:j^{2}:1\right],\left[j:j:1\right],\left[j^{2}:1:1\right]\right\}.
\end{array}
\]
Let us now consider $\alpha\in\mathbb{Q}\left(j\right)\backslash\left\{ 1,j,j^{2},\infty\right\} $
and $S$ a generic invariant algebraic curve of $\mathcal{P}_{\alpha}$.
The intersection of $S$ with the singular locus contains only radial
singularities, because the set of algebraic invariant curves passing
through the non-degenerated singularities is finite. Let us denote
$N_{i}$ the number of times that $S$ goes through the singular point
$s_{i}$. This number does not depend on the choice of $S$. In view
of the description of the group of automorphisms of $\mathcal{P}_{\alpha}$,
for any point $\tilde{s}_{i}$ in the same orbit of $s_{i}$, $S$
goes through $\tilde{s}_{i}$ also $N_{i}$ times. The following lemma
is easy:
\begin{lem}
Let $\mathcal{R}$ be the germ of radial foliation $x\d y-y\d x$
and $S$ be a union of $N$  invariant curves. Let $F$ be an adapted
balanced equation. Then
\[
\left(S,\left(F\right)_{\infty}\right)-\left(S,\left(F\right)_{0}\backslash S\right)=N^{2}-2N.
\]

\end{lem}
\noindent
Therefore, according to Theorem~\ref{thm-poincare-dic}, we have
\begin{eqnarray*}
d_{0}^{2}-6d_{0} & = & \sum_{s\in \sing(S)(\mathcal{P}_{\alpha})\cap S}
(S,(F_{s})_{\infty})-(S,(F_{s})_{0}\backslash S)\\
 & =3 & \sum_{i=1}^{4}N_{i}^{2}-2N_{i}.
\end{eqnarray*}
Thus, we find
\begin{equation}
\label{eq-deg-pencil}
d_{0}=3+\sqrt{9+3\sum_{i=1}^{4}N_{i}^{2}-2N_{i}} \ .
\end{equation}
Now, in view of the description in \cite{mcquillan2001},
the numbers $N_{i}$ can be computed in the following way: let us write $\alpha=\alpha_{1}/\beta_{1}$,
where $\alpha_{1},\ \beta_{1}\in\mathbb{Z}\left[j\right]$ are relatively prime.
Then
\[
N_{1}=N\left(\alpha_{1}\right),\ N_{2}=N\left(\beta_{1}\right),\ N_{3}=N\left(\alpha_{1}-\beta_{1}\right)\textup{ and }N_{4}=N\left(\alpha_{1}+j\beta_{1}\right) ,
\]
where $N\left(a+jb\right)=a^{2}+b^{2}-ab.$ Therefore, the radicand in equation~\eqref{eq-deg-pencil} is
\begin{eqnarray*}
9+3\sum_{i=1}^{4}N_{i}^{2}-2N_{i} \ = \
9+3\left(N\left(\alpha_{1}\right)^{2}-2N\left(\alpha_{1}\right)+N\left(\beta_{1}\right)^{2}-
2N\left(\beta_{1}\right)\right) \\
  +3\left(N\left(\alpha_{1}-\beta_{1}\right)^{2}-2N\left(\alpha_{1}-\beta_{1}\right)+
 N\left(\alpha_{1}+j\beta_{1}\right)^{2}-2N\left(\alpha_{1}+j\beta_{1}\right)\right)
\end{eqnarray*}
This gives
 \[d_{0} = N\left(\alpha_{1}\right)+N\left(\beta_{1}\right)+N\left(\alpha_{1}-\beta_{1}\right)+N\left(\alpha_{1}+j\beta_{1}\right)
,\]
which is consistent with \cite{puchuri2013}.

%\begin{eqnarray*}
%d_{0} & = & 3+\sqrt{\begin{alignedat}{1}9+3\left(N\left(\alpha_{1}\right)^{2}-2N\left(\alpha_{1}\right)+N\left(\beta_{1}\right)^{2}-2N\left(\beta_{1}\right)\right)\\
%\qquad+3\left(N\left(\alpha_{1}-\beta_{1}\right)^{2}-2N\left(\alpha_{1}-\beta_{1}\right)+N\left(\alpha_{1}+j\beta_{1}\right)^{2}-2N\left(\alpha_{1}+j\beta_{1}\right)\right)
%\end{alignedat}
%}\\
% & = & N\left(\alpha_{1}\right)+N\left(\beta_{1}\right)+N\left(\alpha_{1}-\beta_{1}\right)+N\left(\alpha_{1}+j\beta_{1}\right)
%\end{eqnarray*}
%which is consistent with \cite{puchuri2013}.

\section{Topologically bounded invariants of a singularity}

\label{section-topologically-bounded}

The singularity $px\dd y-qy\dd x = 0$, where $p$ and $q$ are coprime
non-negative numbers, can be desingularized by a blowing-up process
following Euclid's algorithm applied to the couple $\left(p,q\right)$. The algebraic multiplicity of the non-isolated separatrix
is $\min\left\{ p,q\right\}$, whereas its Milnor number
  is $1$ for any $p$ and $q.$ Thus, neither the length
of the whole reduction process of a foliation nor the algebraic multiplicity
of a general invariant curve can be bounded by a function of the Milnor
number. This remark brings us to introduce the following definition:

\begin{defn}
Let ${\rm Fol}(\mathbb{C}^{2},0)$  denote the set germs of singular foliations
at $(\mathbb{C}^{2},0)$.
A \emph{numerical datum} in ${\rm Fol}(\mathbb{C}^{2},0)$ is a function
  $\mathfrak{N}: {\rm Fol}(\mathbb{C}^{2},0) \to \mathbb{Z}$.
We say that  the numerical datum $\mathfrak{N}$ is  \emph{topologically bounded}
if there exists a function $\psi:\mathbb{N} \to \mathbb{N}$ such that,
for any  foliation $\mathcal{F}$,
\[
\left|\mathfrak{N}\left(\mathcal{F}\right)\right|\leq\psi\left(\mu_{0}\left(\mathcal{F}\right)\right),
\]
where $\mu_{0}\left(\mathcal{F}\right)$ is the Milnor number of $\mathcal{F}.$
\end{defn}

Notice that the sum of two topologically bounded numerical data is topologically
bounded. The algebraic multiplicity of a foliation is topologically
bounded, since
\[
\nu_{0}\left(\mathcal{F}\right) \frac{\left( \nu_{0} \left(\mathcal{F}\right)+1\right)}{2}\leq \mu_{0}\left(\mathcal{F}\right).
\]
Indeed, if $\mathcal{F}\in\textup{Fol}\left(\mathbb{C}^{2},0\right)$
is given by $\omega=a\textup{d}x+b\textup{d}y$ then
\[
\mu=\dim_{\mathbb{C}}\frac{\mathbb{C}\left\{ x,y\right\} }{\left(a,b\right)}=\dim_{\mathbb{C}}\frac{\mathbb{C}\left\{ x,y\right\} }{\left(x,y\right)^{\nu}}\times\frac{\left(x,y\right)^{\nu}}{\left(a,b\right)}\geq\dim_{\mathbb{C}}\frac{\mathbb{C}\left\{ x,y\right\} }{\left(x,y\right)^{\nu}}=\frac{\nu\left(\nu+1\right)}{2},
\]
where $\mu=\mu_{0}\left(\mathcal{F}\right)$ and $\nu=\nu_{0}\left(\mathcal{F}\right)$.
However,  as illustrated by the above example,  the length of the reduction process is not topologically bounded. In \cite{corral2006},
N. Corral and P. Fern\'andez proved the following:
\begin{thm}
\label{thm:Let--be}
Let $\mathcal{F}$ be a germ of foliation at $(\mathbb{C}^{2},0)$
and $S$ be the union of its formal isolated separatrices. Then the
algebraic multiplicity of $S$ and the length of its reduction process
are topologically bounded.
\end{thm}
In this section, we intend to extend this result to a wider class
of invariant curves. This is the content of:

\begin{maintheorem}
\label{thm:Let--be-1} Let $\mathcal{F}$ be a germ of foliation in
$\mathbb{C}^{2}$. For a fixed $r \in \mathbb{N}$,  let $S$ be   formed by the union of the
following curves:
\begin{enumerate}[(a)]
 \item all  formal isolated separatrices;
 \item one copy of    non-isolated separatrix attached to each dicritical component
of valence one;
 \item $r$ copies of  non-isolated separatrices attached to
each dicritical component of valence three or greater.
\end{enumerate}
Then the
algebraic multiplicity of $S$ and the length of its reduction process
are topologically bounded.
\end{maintheorem}

\subsection{Topological boundedness for saddle-nodes singularities. }

In this section, we are going to prove the following lemma:
\begin{lem}
\label{lem:Let--be} The following   data, concerning singularities appearing along the reduction process, are topologically
bounded:
\begin{enumerate}[(a)]
\item  Milnor numbers;
\item The number of saddle-node and their weak indices.
\end{enumerate}
\end{lem}
\begin{proof}
Let $E_{1}$ be a single blow-up at the singular point $p$ and let $D_{1} = E_{1}^{-1}(p)$.
The   following formulas are classical (see \cite{mattei1980}):
\begin{equation}
\label{eq:1-1}
\mu_{p}\left(\mathcal{F}\right) \ =\
 \begin{cases}
 \displaystyle{ \nu_{p}^{2}\left(\mathcal{F}\right)-\nu_{p}\left(\mathcal{F}\right)-1+\sum_{p'\in D_{1}}\mu_{p^{'}}\left(E_{1}^{*}\mathcal{F}\right)
 \textup{ if  $E_{1}$ is non-dicritical} }\medskip \\
 \displaystyle{  \nu_{p}^{2}\left(\mathcal{F}\right)+\nu_{p}\left(\mathcal{F}\right)-1+\sum_{p'\in D_{1}}\mu_{p^{'}}\left(E_{1}^{*}\mathcal{F}\right)
  \textup{ else.} }
  \end{cases}
\end{equation}
Milnor numbers
usually decrease
with  blow-ups. Suppose however that, in the reduction process of $\mathcal{F}$,    there is a blow-up for which this fails.
We focus at  the first moment  this happens, denoting by $E$ the sequence of blow-ups
  done so far and  by $c$ the singular point of $E^{*}\mathcal{F}$ that, once blown-up,
produces a singularity with higher or equal Milnor
number. According to formulas \eqref{eq:1-1}, such blow-up has to
be non-dicritical and the algebraic multiplicity must satisfy
\[
\nu_{c}\left(E^{*}\mathcal{F}\right)=1.
\]
$E$ is an intermediate step in the minimal reduction process, thus  $c$ is not a reduced singularity.   Thus, looking at all possible cases, we are limited to the following two:
\begin{enumerate}[(i)]
\item $\left(E^{*}\mathcal{F}\right)_{c}$ is given by $\omega=px\textup{d}y-qx\textup{d}x+\cdots$.
In that case,  all Milnor numbers  in the reduction
process are $1$ or less.
\item $\left(E^{*}\mathcal{F}\right)_{c}$ is nilpotent and, following \cite{cerveau1988}, is given locally
by
\[
\textup{d}\left(y^{2}+x^{n}\right)+x^{p}U\left(x\right)\textup{d}y.
\]
Here $\mu_{c}\left(\mathcal{F}\right)=n-1$. One blow-up yields
a singular point with algebraic multiplicity $2$ and Milnor
number $n$. After this point, Milnor numbers decrease.
Anyhow,  along the
 blowing-up process starting at $c$, they cannnot exceed $n$.
\end{enumerate}
We therefore conclude that, along
the reduction process of $\mathcal{F}$,   Milnor numbers     cannot exceed $\mu_{0}\left(\mathcal{F}\right)+1.$

For the second part of the lemma, consider the formal normal form
of  a saddle-node singularity:
\[
x^{p+1}\dd y+y\left(p-\lambda x^{p}\right)\dd x,\ p\geq1.
\]
Its Milnor number  and   its weak index  are both $p+1$, which gives the topological
boundedness for the latter. Now, let   $N$ be the number of saddle-nodes
 in the  reduction of $\mathcal{F}$. We are going
to prove by induction on the length of the reduction process that
\[
N \leq \mu_{0}\left(\mathcal{F}\right).
\]
If the length is $0$, then the singularity is reduced. Thus $N \leq 1 \leq \mu_{0}\left(\mathcal{F}\right)$. Let $k \geq 1$ and
suppose  that the proposition is proved for foliations with    reduction process
of length smaller than $k$. Consider a foliation $\mathcal{F}$
with reduction process  of length $k$. Blow-up  $\mathcal{F}$ once.
The resulting foliation $E_{1}^{*}\mathcal{F}$ has a finite number of
singularities $c_{1},\ldots,c_{p}$. For any $i = 1,\ldots,p$, the local foliation $\left(E_{1}^{*}\mathcal{F}\right)_{c_{i}}$
has a   reduction process  of length less than $k$, giving rise to $n_{i}$
saddle-node singularities. Evidently  $\sum_{i=1}^{p}n_{i}=N$. Besides, by induction, $n_{i} \leq \mu_{c_{i}}\left(E_{1}^{*}\mathcal{F}\right)$ for $i = 1,\ldots,p$. We have:
\begin{enumerate}[(i)]
\item If $\nu_{0}\left(\mathcal{F}\right)\geq2$ then, according to \eqref{eq:1-1},
\[
\mu_{0}\left(\mathcal{F}\right) \geq \sum_{i=1}^{p} \mu_{c_{i}}\left(E_{1}^{*}\mathcal{F}\right) \geq \sum_{i=1}^{p}n_{i}=N.
\]
\item If $\nu_{0}\left(\mathcal{F}\right)=1$, then either $\mathcal{F}$ is of the form
$\omega=px\textup{d}y-qx\textup{d}x+\cdots$, and no
 saddle-nodes appear in its reduction, or $\mathcal{F}$
is nilpotent. In the second case, $\mu_{0}\left(\mathcal{F}\right)\geq2$
and $N\leq1$ and  the inequality also holds.
\end{enumerate}
\end{proof}

\subsection{Dicritical components of valences $\val(D)=1$ and $\val(D) \geq 3$}

We have the following properties concerning  dicritical components
of valence one:
\begin{lem}
\label{lem:The-multiplicity-of}
The number and the multiplicities of the dicritical components of
valence one are topologically bounded. \end{lem}
\begin{proof}
We perform a blowing-up process  $E$ in order to
reduce all formal isolated separatrices of $\mathcal{F}$ . According
to Theorem~\ref{thm:Let--be}, the number of blow-ups composing $E$ is topologically
bounded. Then,  the dicritical components
of valence $1$ in the exceptional divisor of $E$ have topologically bounded
multiplicities.

The remaining dicritical
singularities of $E^{*}\mathcal{F}$ are of special kind: they are crossed by
at most one  isolated separatrix of $\mathcal{F}$,  which are smooth and transversal to the divisor $\DD = E^{-1}(0)$.
If  $ c \in \DD$ is a singularity of $E^{*}\mathcal{F}$, then the reduction process of $\left(E^{*}\mathcal{F}\right)_{c}$
cannot have dicritical components of valence three or greater.
To see this, we  use the  fact that each invariant connected component of the reduction divisor   carries
 an isolated separatrix, as done in the proof of  Lemma~\ref{lemma_pure_multiplicity}. If such a dicritical component
existed, then we would find isolated separatrices outside
 $\DD$, one or two, respectively if $c$ is a corner or a non-corner point of $\DD$. This is a contradiction,
 since  all isolated separatrices of $\mathcal{F}$ are reduced.

The reduction process of $\left(E^{*}\mathcal{F}\right)_{c}$ thus
contains only dicritical components of valences $1$ or $2$ and at most one isolated separatrix of $\mathcal{F}$
 attached to some invariant component of multiplicity
one. The expression for the algebraic multiplicity given by Proposition~\ref{prop:Equa-Ba} and formula~\eqref{eq-tau}
lead to
\[
\nu_{c}\left(\left(E^{*}\mathcal{F}\right)_{c}\right) = \sum_{D\textup{ dicritical, } v(D) = 1} \rho \left(D\right)+\left(\cdots\right),
\]
where the dots is an expression that depends only on the saddle-node
singularities and their weak indices. Since these data are topologically
bounded, this formula   ensures that the number and multiplicities of dicritical
 components
  of valence one  are topologically bounded.
\end{proof}

In what concerns dicritical components of
higher valences, we achieve the following:

\begin{cor} {\rm (of the proof)}
\label{cor:The-multiplicity-of}
All dicritical components of
valence three or greater appear after a topologically bounded number of blow-ups.
Besides,
their
 number,  multiplicities and   valences are also topologically bounded.
\end{cor}
\begin{proof}
All dicritical components of valence at least three   appear  as soon as all isolated separatrices are desingularized.
It then follows from Theorem~\ref{thm:Let--be} that a topologically bounded number of blow-ups create them.
This is enough to bound their number and multiplicities. For their valences, it suffices to notice that, for
a dicritical component $D$,
$\val(D) \leq \iota(D) + 2$, where $\iota(D)$ is the number of isolated separatrices originating
on $D$.
\end{proof}

\subsection{Proof of Theorem~\ref{thm:Let--be-1}.}

\begin{proof}(Theorem~\ref{thm:Let--be-1})
We observe that, in   formulas \eqref{eq:1-1}, if
$\nu_{p}\left(\mathcal{F}\right)>2$ then both second degree expressions
in $\nu_{p}\left(\mathcal{F}\right)$  provide   strictly positive integers.
Let us then apply these formulas inductively  along the reduction process of $\mathcal{F}$ as long
as  the blown-up
singularities   have algebraic multiplicities at least  $2$. In so  doing
 we are led to an expression of the form
\[
\mu\left(\mathcal{F}\right)=\left(\cdots\right)+\sum_{c\in \DD}\mu_{c}\left(E^{*}\mathcal{F}\right),
\]
where $E$ is the blowing-up process and $\DD = D^{-1}(0)$ is the exceptional divisor. The expression  represented by $\left(\cdots\right)$ is a positive integer, greater or equal to
the number of blow-ups in $E$. In particular,  this number
is at most $\mu\left(\mathcal{F}\right)$
and, thus, topologically bounded. Let $c\in \DD$ and $S_{c}$ be
the local component of the strict transform of $S$ by $E$. The multiplicity
$\nu_{c}\left(E^{*}\mathcal{F}\right)$ is equal to $0$ or $1$.
Thus, the foliation $\left(E^{*}\mathcal{F}\right)_{c}$ belongs to
the following list:
\begin{enumerate}[(i)]
\item A regular foliation.

\item A simple singularity with two eigenvalues whose quotient is   a
not a positive rational number, $\omega=\lambda_{1}x\dd y - \lambda_{2}y\dd x+\cdots$, with $ \lambda_{1} / \lambda_{2}\notin\mathbb{Q}^{+}$.

\item A Dulac singularity, $\omega=x\dd y-ny\dd x+\cdots$, with $n\in\mathbb{N}^{*}$.
The normal form is given by $\left(x+ay^{n}\right)\dd y-ny\dd x.$
The foliation is dicritical if and only if $a=0$. If $a\neq0$ then
it has only one regular separatrix.

\item A non-degenerated dicritical singularity, $\omega=px\dd y-qy\dd x+\cdots$,
with $p,q\in\mathbb{N}^{*}$ and $p / q \notin\mathbb{N}\cup(1 / \mathbb{N}).$

\item A nilpotent singularity, $\omega=y\dd y+\cdots$.
\end{enumerate}

If the singularity of  $\left(E^{*}\mathcal{F}\right)_{c}$ provides only isolated separatrices in $S_c$, then according
to Theorem \ref{thm:Let--be}, the number of blow-ups in their reduction process is topologically bounded. From previous lemmas and results about simple, Dulac, resonnant and nilpotent singularities in
\cite{berthier1999,cerveau1988,meziani1996,meziani2007},
the local curve $S_{c}$ possibly contains a dicritical separatrix only in the following three cases:
%So, it remains to analyse the following three cases:

\begin{center}
\includegraphics[scale=0.6]{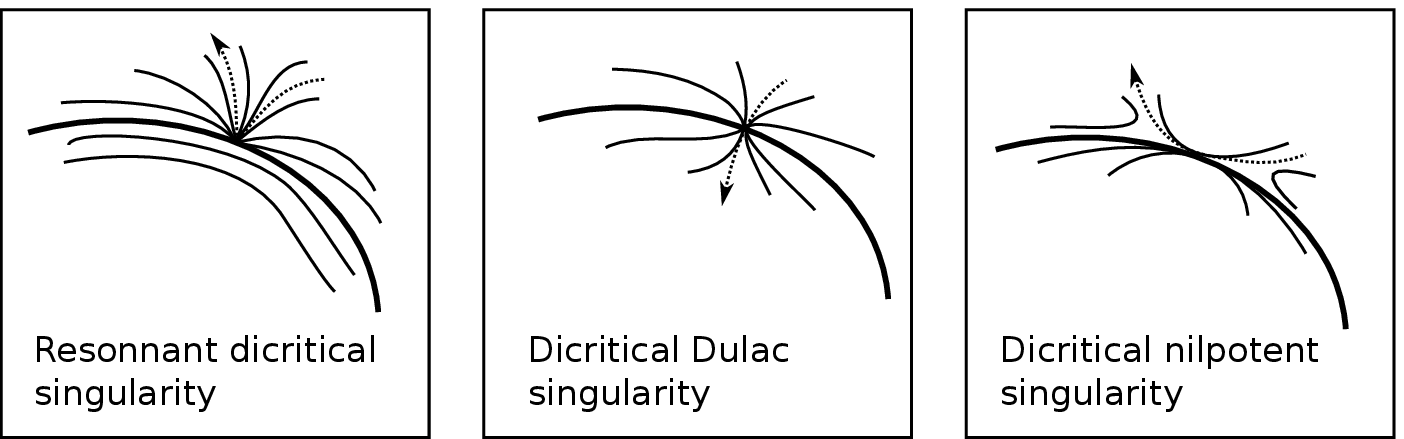}
\par\end{center}

Their analysis depends on whether or not $c$ is a corner of $\DD$.

 \smallskip
\par \noindent \underline{\it{First case:}} \  $c \in \DD$ is a non-corner point.
Denoting by  $D$ the local irreducible component of $\DD$ at $c$, we have the following
possibilities:

 \smallskip \par (i)   $\left(E^{*}\mathcal{F}\right)_{c}$ is a resonant dicritical
singularity of the form $px\dd y-qy\dd x$ with $ p /q \notin\mathbb{N}\cup (1/\mathbb{N})$.
Then all leaves are singular, except those given by $x=0$ and $y=0$. Thus,
locally $D$ is   defined by $x=0$, by   $y=0$ or is transverse
to the foliation. In all cases, the reduction process will produce
a unique dicritical component of valence $2$. So the local trace of
$S_{c}$ reduces to $x=0$ or to $y=0$, which are isolated separatrices.

 \smallskip \par (ii)   $\left(E^{*}\mathcal{F}\right)_{c}$ is a dicritical Dulac singularity
 of the form $x\dd y-ny\dd x$. Then  all leaves are smooth,
given by $x=0$ or by $y=\alpha x^{n}$, with $\alpha\in\mathbb{C}$.
If $D$ is locally given by $x=0$, then the local trace of $S_{c}$ is
some leaf $y=\alpha x^{n},$ which is smooth transverse to $\DD$ and thus
reduced. If $\DD$ is of the form $y =\alpha x^{n}$, then either the
local trace of $S_{c}$ is $x=0$, hence reduced, or it is given by
$y=\beta x^{n}$ with $\beta\neq\alpha$. In the latter case, the
reduction process provides a dicritical component of valence $2$,
which is not allowed in  definition of $S$. Suppose finally
that $D$ is not invariant. If it is transverse to every leaf or tangent
to $x=0$, then the local trace of $S$ is written as $x\left(y-\alpha x^{n}\right) = 0$
for some $\alpha$ and is reduced after one blow-up. If it is tangent
to $y=0$, then two cases can occur: if $D$ is tangent to $y=0$
with an order  $n$ or greater, then the reduction process of $\left(E^{*}\mathcal{F}\right)_{c}$
produces a dicritical component of valence $2$. If $D$ is tangent
to $y=0$ with an order $p<n$, then the dicritical component $D_{0}$
appearing in the reduction process  of $\left(E^{*}\mathcal{F}\right)_{c}$
is of valence $1$. However, since the multiplicity is an additive
function, we have
\[
\rho \left(D_{0}\right)=p \, \rho\left(D\right)\geq p.
\]
Lemma \ref{lem:The-multiplicity-of} ensures that $\rho \left(D_{0}\right)$
is topologically bounded and so is the integer $p$. Now, after $p+1$
blow-ups the curves $S_{c} =\left\{ y=\alpha x^{n}\right\} $ and $D$
are separated. So, $S_{c}$ becomes reduced after a topologically bounded number
of blow-ups.

 \smallskip \par (iii)   $\left(E^{*}\mathcal{F}\right)_{c}$ is a dicritical nilpotent
singularity. Then after a topologically bounded number of blow-ups,
we are led to a resonant or Dulac dicritical singularity $c'$. Indeed,
according to \cite{meziani1996,meziani2007}, in this situation, $2p=n=\mu\left(\left(E^{*}\mathcal{F}\right)_{c}\right)$,
so the integer $p$ is topologically bounded. If
  $c'$ is a regular point of the   exceptional divisor,
then we are in  situations (i) or (ii) above.
The case where  $c'$ is a corner will be treated
below.
%\end{enumerate}

\par \noindent \underline{\it{Second case:}}  $c \in \DD$ is at a  corner.
Now we denote locally $\left(\DD\right)_{c}=D_{1}\cup D_{2}$. The following
subcases occur:

 \smallskip \par (i)
  $\left(E^{*}\mathcal{F}\right)_{c}$ is a resonant dicritical
singularity. The proof goes  as in First case, (i).   All leaves except $x=0$ and $y=0$ are singular and
the reduction process    produces a dicritical component of
valence $2$. Therefore,
$S_{c}$ coincides with one of the isolated separatrices $x=0$ or $y=0$.

 \smallskip \par (ii)  $\left(E^{*}\mathcal{F}\right)_{c}$ is a dicritical Dulac singularity.
Suppose first  that both local components of $\DD$ are invariant. Then  necessarily $\left(\DD\right)_{c}$
has an equation of
  the form $x\left(y-\alpha x^{n}\right)$ for some $\alpha \in \mathbb{C}$. The
reduction process  produces a dicritical component of valence $2$ and  $S_{c}$
contains no dicritical components.

Suppose now that only one local  component of $\left(\DD\right)_{c}$, say
  $D_{1}$, is invariant. If $D_{1}=\left\{ x=0\right\} $
and $D_{2}$ is transverse to $y=0$, then $S_{c}$ is of the form $y=\alpha x^{n}$
for some $\alpha \neq 0$. In this situation, one more blow-up reduces $S_{c}$.
Suppose next that $D_{2}$ is tangent to $y=0$. Then two cases can
occur: if $D_{2}$ is tangent to $y=0$ with   order at least $n$,
then the reduction process of $\left(E^{*}\mathcal{F}\right)_{c}$
produces a dicritical component of valence $2$. On the other hand, if $D_{2}$ is tangent
to $y=0$ with order $p<n$, then the dicritical component $D_{0}$
appearing in the reduction  of $\left(E^{*}\mathcal{F}\right)_{c}$
has valence $1.$ The multiplicity of $D_{0}$ is written
\[
\rho \left(D_{0}\right)=\rho \left(D_{1}\right) + p \rho \left(D_{2}\right)\geq p.
\]
Lemma \ref{lem:The-multiplicity-of} assures that $\rho\left(D_{0}\right)$
is topologically bounded and so is the integer $p$. Now, after $p+1$
blow-ups, the curve $S_{c}=\left\{ y=\alpha x^{n}\right\}$ and $D_{2}$
are separated. So, $S_{c}$ is reduced after a topologically bounded number
of blow-ups.
If $D_{1}=\left\{ y=\alpha x^{n}\right\} $ for some
$\alpha \in \mathbb{C}$, then $D_{2}$ is transverse to $y=0.$ Thus, the reduction
process   produces a dicritical component of valence $2$. Hence
$S_{c}$ is the isolated separatrix $\left\{ x=0\right\}$.

Suppose that both components $D_{1}$ and $D_{2}$ are non-invariant.
As before, we have the following alternatives: if neither $D_{1}$
nor $D_{2}$ is tangent to $y=0$ then one blow-up is enough to reduce
$S_{c}.$ If one of the components is tangent to $y=0$, then either the
order of tangency is greater or equal to $n$ and a   dicritical
component  of valence 2 appears, or it is smaller than
$n$ and, then, topologically bounded.

 \smallskip \par (iii)  $\left(E^{*}\mathcal{F}\right)_{c}$ is a dicritical nilpotent
singularity. Then after a topologically bounded number of blow-ups,
we are led either to a resonant or to a Dulac dicritical singularity. These cases
were already treated.

\smallskip

This ends the proof of Theorem \ref{thm:Let--be-1}.
\end{proof}

\subsection{Consequences to the Poincaré problem}

Theorems \ref{thm-poincare-dic} and \ref{thm:Let--be-1} allow us to give an answer to the Poincaré problem
in the dicritical case which generalizes the one in \cite{corral2006}:

\begin{maintheorem}
\label{thm-poincare-dicI}
For every $d,r \in \mathbb{N}$
there exists $N = N(d,r) \in \mathbb{N}$ with the following property: if
  $\mathcal{F}$ is an analytic foliation on  $\mathbb{P}^{2}_{\mathbb{C}}$   of degree
$d$ having an  invariant algebraic curve $S$ of degree $d_{0}$ such
   that, for every $p \in \sing(\cl{F}) \cap S$, the local component    $(S)_{p}$ is composed by:
\begin{enumerate}[(a)]
 \item any number of isolated separatrices;
 \item at most one   separatrix attached to each dicritical  component
of valence one;
\item  no branches of separatrices attached to dicritical components of valence two;
 \item at most $r$ branches of separatrices attached to
each dicritical component of valence three or greater;
 \item if $p$ is a radial singularity, at most $r$ branches of separatrices attached to
the dicritical component of valence zero.
\end{enumerate}
Then   $d_{0} \leq N$.
\end{maintheorem}
\begin{proof} We start by remarking that   Bezout's Theorem for foliations,
\[ \sum_{p \in \sing(\cl{F})} \mu_{p}(\cl{F}) = d^{2} + d + 1,\]
gives  a bound for the number of points $p \in \sing(\cl{F})$ as well as  for each $\mu_{p}(\cl{F})$      in terms of $d$.  Besides, Theorem~\ref{thm:Let--be-1} gives that the reduction process of $(S)_{p}$ has a topologically bounded length.
All we have to do is to find a topological bound to the term $(S,(\hat{F}_{p})_{\infty})_{p}$ in Theorem \ref{thm-poincare-dic}, where
$\hat{F}_{p}$ is a balanced equation adapted to the invariant curve $S$ at $p$.
 Hypothesis \emph{(b)} and \emph{(c)} on the   structure of $(S)_{p}$ imply that all branches in $(\hat{F}_{p})_{\infty}$ are associated to dicritical components of valence
$v(D) \geq 3$. By Lemma~\ref{cor:The-multiplicity-of}, these components arise after a topologically bounded number of blow-ups, their number and their valences are also topologically bounded.
This is enough to conclude that the number of branches and the   length of the reduction process of  $(\hat{F}_{p})_{\infty}$ are both topologically bounded. In the same way,
 $S_{p} \cup (\hat{F}_{p})_{\infty}$ has a reduction process of topologically bounded length.  This is enough to bound $(S,(\hat{F}_{p})_{\infty})_{p}$
in terms of $\mu_{p}(\cl{F})$, and thus in terms of $d$.
\end{proof}

In a sense, Theorem $E$ provides the ultimate criterion to obtain
a bound of the degree of an invariant curve from local topological data
associated to the singularities of a foliation. Indeed, the foliations described in sections
\ref{subsection-node} and \ref{subsection-linsneto}  violate the conclusion of theorem $E$. They appear to be
prototype of foliations from which one can construct exemples of families of
foliations having algebraic invariant curves that do not satisfy
conditions (b), (c), (d) or (e).

\section{Topological invariance of the algebraic multiplicity}

\label{section-topological-invariance}

Let $\mathcal{F}_{1}, \mathcal{F}_{2} \in {\rm Fol}(\mathbb{C}^{2},0)$ be two local analytic foliations. We say that a germ of homeomorphism $\Phi: (\mathbb{C}^{2},0) \to (\mathbb{C}^{2},0)$
is a topological equivalence between $\mathcal{F}_{1}$ and $\mathcal{F}_{2}$ if $\Phi$ takes
leaves of $\mathcal{F}_{1}$ into leaves of $\mathcal{F}_{2}$.
The Milnor number  (see \cite{camacho1984}) of foliation and  the $GSV$-index   (see \cite{gomezmont1991})
are well known topological invariants.
On the other hand, the topological invariance  of the algebraic multiplicity   is so far not known. This is true when  $\mathcal{F}$ is either a generalized curve foliation (by \cite{camacho1984})
--- in which case all separatrices are convergent --- or
a non-dicritical second type foliation with only  convergent separatrices (by \cite{mattei2004}).

Let $\Phi: (\mathbb{C}^{2},0) \to (\mathbb{C}^{2},0)$
be a topological equivalence between   foliations $\mathcal{F}_{1}, \mathcal{F}_{2} \in {\rm Fol}(\mathbb{C}^{2},0)$.
It is clear that if $S$ is a convergent
separatrix of $\mathcal{F}_{1}$ then $\Phi_{*}S := \Phi(S)$ is  a convergent separatrix for $\mathcal{F}_{2}$.
Actually, $\Phi(S)$ is an analytic curve as a consequence of Remmert-Stein Theorem, since
 $\overline{\Phi(S \setminus \{0\})} = \Phi(S \setminus \{0\}) \cup \{0\}$.
We recall that   non-convergent separatrices
appear as the weak separatrices of non-tangent saddle-nodes and thus are   all isolated separatrices. On the other hand,   dicritical separatrices   converge.

We start by a remark: a topological equivalence   respects the ``dicritical structure'' of the desingularization.
This means that the same combinatory  of blow-ups that desingularizes the dicritical separatrices of $\cl{F}_{1}$, when applied to $\cl{F}_{2}$,  will produce  dicritical components at exactly
the same positions. This is a consequence of Zariski's Equidesingularization Theorem for curves (see for instance
 \cite{brieskorn1986}). Actually, if $D$ is a dicritical component of the desingularization of $\cl{F}_{1}$,
 it suffices to take two dicritical separatrices $S_{1}$ and $S_{2}$ attached to $D$.
Zariski's Theorem gives that they correspond to separatrices $S_{1}'$ and $S_{2}'$ of $\cl{F}_{2}$ which will be attached to a component $D'$ produced by the same sequence of blow-ups as $D$. This argument
actually works to any pair of curves passing through $D$, thus $D'$ will be crossed by infinitely many separatrices
and will be dicritical. We will call $D$ and $D'$ {\em equivalent} dicritical components. This term is also justified
by the following:

\begin{prop} Let $\cl{F}_{1}$ and $\cl{F}_{2}$ be topological equivalent foliations. Let $D$ and $D'$ be equivalent dicritical components in their desingularizations.
Then $D$ and $D'$ have the same valence.
\end{prop}
\begin{proof}
We look at $D$ and $D'$ at the very moment they appear in the desingularization process.
Denote by  $\tilde{\cl{F}}_{1}$ and $\tilde{\cl{F}}_{2}$ the transforms of $\cl{F}_{1}$ and $\cl{F}_{2}$
at this step. We first notice that
if $p \in D$ is  regular for $\tilde{\cl{F}}_{1}$ then the Equidesingularization
Theorem implies that the leaf   $L$   containing
$p$ corresponds, by the topological equivalence, to a smooth curve $L'$ crossing $D'$ transversely  at a point $p'$.
The same occurs for all points near $p$ and this forces $q$ to be a regular point
for $\tilde{\cl{F}}_{2}$.
Let us first prove that $\val(D) \leq \val(D')$.
This is evidently true if $\val(D) = 0$.
If $\val(D) = 1$, since  $D'$ appears applying the same sequence of blow-ups that produce $D$, it arises
from a blow-up at a non-corner point. Thus $\val(D')$ is 1 at least. The same argument works when
$\val(D) = 2$, now $D'$ also arising from the blow-up at a non corner point.
Let us  now suppose that $\val(D) > 2$.
In the case in which both $D$ and $D'$ appear from the blow-up at a non corner point,
 in order to
proceed towards the desingularization of $\cl{F}_{1}$, we will start a sequence of blow-ups
at $\val(D) - 1$ different points in $D$, which either  correspond  to singular  points of $\tilde{\cl{F}}_{1}$ over $D$ or
 to points where $\tilde{\cl{F}}_{1}$  and $D$ are tangent. In the first case, a convergent separatrix
exists by the separatrix Theorem \cite{camacho1982}. In the second, the tangent leaf   gives rise
 to a convergent separatrix. Following these separatrices by the topological equivalence, we   find
 $\val(D) - 1$ points over $D'$ where $\tilde{\cl{F}}_{2}$ is not desingularized. We conclude therefore that $\val(D) \leq \val(D')$. The same  argument works when $D$ and $D'$ arise from the blow-up at a corner point.
 In all cases, we have $\val(D) \leq \val(D')$. Finally, inverting the roles of $D$ and $D'$, we also have
 $\val(D') \leq \val(D)$ and this   finishes the proof of the proposition.
\end{proof}

The main drawback in dealing with topological equivalences is that they do not track formal separatrices as these are not realizable geometric objects.
Thus, when considering topological equivalent foliations $\cl{F}_{1}$ and
$\cl{F}_{2}$,  it is reasonable to suppose that
all  separatrices for at least one of the foliations, say $\cl{F}_{1}$, are convergent. With this hypothesis, given  a balanced equation of
separatrices $F_{1}$   for $\cl{F}_{1}$, we can choose a balanced equation   $F_{2}$   for $\cl{F}_{2}$ in such a way that   $\Phi_{*}(F_{1})_0  \subset (F_{2})_0$
and $\Phi_{*}(F_{1})_\infty  = (F_{2})_\infty$, the inclusion being an equality
when all separatrices of $\cl{F}_{2}$ are convergent. We can also formulate an analogous statement   for
balanced equations adapted to   sets of separatrices $C$ of $\cl{F}_{1}$ and $\Phi_{*}C$ of $\cl{F}_{2}$.

Now, our idea is
 to explore the formula in Theorem~\ref{thm-gsv-polar}. Let $C_{1}$ be a set of   separatrices
 of $\cl{F}_{1}$ --- with only convergent separatrices --- and
  $F_{1}$ be a $C_{1}$-adapted  balanced equation. Let $F_{2}$ be a balanced set of separatrices for $\cl{F}_{2}$
  adapted to
    $C_{2} = \Phi_{*}C_{1}$ as described in the previous paragraph.
  In Theorem~\ref{thm-gsv-polar},   we denote
\[ \delta_{0}(\cl{F}_{i},C_{i}) = -\left[\left(C_{i},\left(F_{i}\right)_{0}\backslash C_{i}\right)_0 -\left(C_{i},\left(F_{i}\right)_{\infty}\right)_0\right],\ i = 1,2.\]

With the notations above, the next result is straightforward:

\begin{prop}
\label{prop-delta}
If all separatrices of $\cl{F}_{1}$ are convergent, then
\[ \delta_{0}(\cl{F}_1,C_{1})  \geq \delta_{0}(\cl{F}_2,C_{2}) .\]
Moreover, equality holds
if and only if  if all separatrices of $\cl{F}_{2}$ are convergent.
\end{prop}

The next result, a consequence  of Theorem~\ref{thm-gsv-polar},  says in particular
that   the $\var$-index is a topological invariant when both foliations have convergent separatrices.
\begin{prop}
\label{prop-top-invariance}
Suppose that all separatrices of $\cl{F}_{1}$ are convergent and let $C_1$ be a set of some of its separatrices.  Then
\[ \var_{p}(\cl{F}_1,C_{1})  \leq  \var_{p}(\cl{F}_2,C_{2}) ,\]
where   $C_{2} = \Phi_{*}C_{1}$. Furthermore, equality holds
if and only if  all separatrices of $\cl{F}_{2}$ are convergent.
\end{prop}
\begin{proof} The result follows straight  from the topological invariance of the $GSV$-index  and
from Proposition \ref{prop-delta}.
\end{proof}

We close this article by presenting a result
on the topological invariance of the algebraic multiplicity. It
 generalizes at a time the results contained in
\cite{camacho1984}  and in
 \cite{mattei2004} mentioned in the beginning of the section.

\begin{maintheorem}
\label{thm-topological-invariance}
 Suppose that $\mathcal{F}_{1}$ and $\mathcal{F}_{2}$ are topological equivalent  analytic foliations having
only convergent separatrices. Then $\cl{F}_{1}$ is of second type if
and only if $\cl{F}_{2}$ is of second type. As a consequence, $\cl{F}_{1}$ and
  $\cl{F}_{2}$ have the same algebraic multiplicities.
\end{maintheorem}
\begin{proof}
We start by remarking that a local foliation always has a separatrix which  does
not arise as the weak separatrix of a saddle-node. This is the essence of  the proof
of the Separatrix Theorem in \cite{camacho1982}. Let $B_{1}$ be such a separatrix for
 $\cl{F}_{1}$ and set
$B_{2} = \Phi_{*}B_{1}$.  If $\cl{F}_{1}$ is of second type then, as a consequence
of   Proposition~\ref{prop:Var1}, $\var_{p}(\cl{F}_1,B_{1})$
does not change along the desingularization of $B_{1}$ and, using Example~\ref{example-reduced}, we find $\var_{p}(\cl{F}_1,B_{1}) = 0$.  By Proposition~\ref{prop-top-invariance}, we also have
$\var_{p}(\cl{F}_2,B_{2}) = 0$. Proposition~\ref{prop:Var1} then gives $\tau(\cl{F}_{2}) = 0$, which means
that $\cl{F}_{2}$ is a foliation of second type.
Now,
 since a balanced set of separatrices for $\cl{F}_{1}$ is taken by $\Phi$
into a balanced set for $\cl{F}_{2}$, both  sets   have the same algebraic multiplicity
 as a consequence of  the Equidesingularization Theorem for curves. Finally, Proposition~\ref{prop:Equa-Ba} assures that the
  same will be true for the algebraic
 multiplicities of  $\cl{F}_{1}$ and
  $\cl{F}_{2}$.
\end{proof}

\bibliographystyle{plain}
\bibliography{Bibliographie}

\medskip \medskip
\noindent
Yohann Genzmer  \\
I.M.T. Universit\'e Paul Sabatier \\
118 Route de Narbonne \\
31400  --
Toulouse,
France\\
yohann.genzmer@math.univ-toulouse.fr

\medskip \medskip
\noindent
Rog\'erio  Mol  \\
Departamento de Matem\'atica \\
Universidade Federal de Minas Gerais \\
Av. Ant\^onio Carlos, 6627  \  C.P. 702  \\
30123-970  --
Belo Horizonte -- MG,
Brasil \\
rsmol@mat.ufmg.br

\end{document}